\documentclass[final,leqno,onefignum,onetabnum]{siamltex1213}
\usepackage{amsfonts}
\usepackage{amsmath}
\usepackage{amssymb}
\usepackage{latexsym}
\title{Nonzero-sum risk-sensitive stochastic differential games\thanks{The work of the first named author is supported in part by UGC Centre for Advanced Study. The work of the second named author is supported in part by the DST, India 
project no. SR/S4/MS:751/12. The work of the third named author is supported in part by Dr. D. S. Kothari postdoctoral fellowship of UGC.}} 
\author{Mrinal K. Ghosh \footnotemark[2]\ \and K. Suresh Kumar \footnotemark[3] \
\and Chandan Pal \footnotemark[2]}

\begin{document}
\maketitle
\slugger{mms}{xxxx}{xx}{x}{x--x}
\renewcommand{\thefootnote}{\fnsymbol{footnote}}

\footnotetext[2]{Department of Mathematics, Indian Institute of Science, Bangalore-12, India. \; (\email{mkg@math.iisc.ernet.in, chandan14@math.iisc.ernet.in})}
\footnotetext[3] { Department of Mathematics, Indian Institute of Technology Bombay, Mumbai-400076, India. \email{suresh@math.iitb.ac.in}}

\begin{abstract}
\noindent We study two person nonzero-sum stochastic differential games with risk-sensitive discounted and ergodic cost criteria. Under certain conditions we establish a Nash equilibrium in Markov strategies for the discounted cost criterion and a Nash equilibrium in stationary strategies  for the ergodic cost criterion. We achieve our results by studying the relevant systems of coupled HJB equations.
\end{abstract}



\pagestyle{myheadings}
\thispagestyle{plain}
\markboth{MRINAL K. GHOSH, K. SURESH KUMAR
 AND CHANDAN PAL }{Risk-sensitive stochastic differential games}

\section{Introduction}

We study nonzero-sum risk-sensitive stochastic differential games on the infinite time horizon. In the literature of stochastic differential games, one usually considers the expectation of the integral of costs \cite{Borkar-Ghosh}, \cite{ED}, \cite{Va} etc. This is the so called risk-neutral situation where the players (i.e., the decision makers or controllers) ignore the risk. If the players are risk-sensitive (i.e., risk-averse or risk-seeking), then the most appropriate cost criterion is the expectation of the exponential of the integral of costs. We refer to \cite{No} for an excellent note on risk-sensitive Nash-equilibria. Since the cost criterion is the expectation of the exponential of the integral costs, it is multiplicative as opposed to the additive nature of the cost criterion in the expectation of the integral costs case. Due to this, the analysis of the risk-sensitive case is significantly different from its risk-neutral counterpart. To our knowledge, the risk-sensitive criterion was first introduce by Bellman \cite{RB}, see \cite{PW} and the reference therein. Though this criterion has been studied extensively for stochastic optimal control problems \cite{BN}, \cite{BFN}, \cite{Anup}, \cite{AnupBorkarSuresh}, \cite{FH}, \cite{FM}, \cite{J}, \cite{MR}, \cite{Na},  \cite{Ru}, the corresponding literature in the context of stochastic differential games is rather limited. Some exceptions are \cite{Ba}, \cite{BG}, \cite{EH}. In this paper we address the existence of Nash equilibria for stochastic differential games where the state of the system is governed by controlled diffusion processes. We consider two risk-sensitive cost evaluation criteria: discounted and ergodic. For both criteria, we establish Nash equilibria under certain conditions.

The rest of this paper is organized as follows. Section 2 deals with the problem description. The discounted cost criterion is analyzed in Section 3. Under certain additive structural conditions on the drift vector and the cost functions we establish the existence of a Nash equilibrium in Markov strategies. In Section 4 we study the ergodic case. Under a Lyapunov type stability condition we establish the existence of a Nash equilibrium in stationary strategies. Section 5 contains some concluding remarks.

\section{Problem Description}
For the sake of notational simplicity we treat two player game. The $N$-player game for $N \geq 3$  is analogous. Let $U_i, \ i =1,2$, be  compact metric spaces and $V_i=\mathcal{P}(U_i)$,  the space of probability measures on  $U_i$ with Prohorov topology. Let $$\Bar{b} = (\Bar{b}_1, \cdots , \Bar{b}_d) 
 : \mathbb{R}^d \times U_1 \times U_2 \to  \mathbb{R}^d, $$ $$  
\Bar{r}_k : \mathbb{R}^d \times U_1 \times U_2 \to [0, \ \infty), \, k=1, 2, $$   $$ \sigma : \mathbb{R}^d \to \mathbb{R}^{d \times d}$$ be  functions satisfying the following: \\

\noindent {\bf (A1)} (i) The functions $\Bar{b}, \sigma, \Bar{r}_k, \, k =1,2$, are bounded, Lipschitz continuous in 
the first argument uniformly over the second and third arguments.
Also the functions $\Bar{b}, \ \Bar{r}_k, \, k =1,2$, are (jointly) continuous. 

(ii) The function $a = \sigma \sigma^{\perp} : \mathbb{R}^d \to \mathbb{R}^{d \times d}$ 
is uniformly elliptic, i.e., the infimum of the eigenvalues of $a$ is strictly positive. Here  for a matrix say
$A$, $A^\perp$ denote the transpose of $A$.\\

Define
$b = (b_1, \cdots , b_d)  : \mathbb{R}^d \times V_1 \times V_2 \to \mathbb{R}^d, 
\ r_k : \mathbb{R}^d \times V_1 \times V_2 \to [0, \ \infty)$ by
\interdisplaylinepenalty=0
\begin{eqnarray*}
b_i(x, v_1, v_2) & = & \int_{U_2} \int_{U_1} \Bar{b}_i(x, u_1, u_2) v_1(du_1) v_2(du_2),  \\
r_k(x, v_1, v_2) & = & \int_{U_2} \int_{U_1} \Bar{r}_k (x, u_1, u_2) v_1(du_1) v_2(du_2), 
  x \in \mathbb{R}^d, \\
&&  \ \ \ \  \ \ v_1 \in V_1, \, v_2 \in V_2, \, i = 1, \cdots , d, \, k =1,2.
\end{eqnarray*}
Consider the following controlled diffusion process given by the solution of the stochastic differential equation (s.d.e.) 
\begin{equation}\label{statedynamics}
d X(t) \ = \ b(X(t), v_1(t, X(t)), v_2(t, X(t))) dt + \sigma(X(t)) dW(t), 
\end{equation}
where $W(\cdot)$ is an $\mathbb{R}^d$-valued standard Wiener process, $v_i : [0, \infty) \times \mathbb{R}^d 
\to V_i, \ i =1, 2$ is a measurable function.  
Under (A1), the s.d.e. (\ref{statedynamics}) has a unique weak solution 
which is a strong Markov process for a given initial condition $X(0) =x$; 
see [\cite{AriBorkarGhosh}, Theorem 2.2.12, p.45] for details. 
For the stochastic differential game, the controlled diffusion given by (\ref{statedynamics}) has the
following interpretation. The $i$th player controls the state dynamics, i.e., the controlled
diffusion given above,  through the choice of his strategy $\Bar{v}_i(t) = v_i(t, X(t)), t \geq 0$. 
By an abuse of notation, the measurable map $v_i : [0, \ \infty) \times \mathbb{R}^d \to V_i$ itself
is called a Markov strategy for player $i$. Let
${\mathcal M}_i \ = \ \{ v_i : [0, \ \infty) \times \mathbb{R}^d \to V_i \; | \; v_i \ {\rm is\ measurable}\}$
be the set of all Markov strategies for player $i$.  If $v_i$ doesn't
have explicit dependence on $t$, i.e., $v_i(t, x) = v_i(x), \ x \in \mathbb{R}^d, \ t \geq 0$, it is
said to be a stationary Markov strategy for player $i$. The set of all stationary Markov strategies
for player $i$ is denoted by ${\mathcal S}_i, \ i =1,2$. We topologize ${\mathcal S}_i, \ i=1,2$, using
a metrizable weak* topology on $L^{\infty}(\mathbb{R}^d ; {\mathcal M}_s(U_i))$,
where ${\mathcal M}_s(U_i)$ denotes the space of all signed measures on $U_i$ with weak* topology.
 Since ${\mathcal S}_i$ is a  subset of the unit ball of $L^{\infty}(\mathbb{R}^d ; {\mathcal M}_s(U_i))$,
it is compact under the above weak* topology. One also has the following characterization of the
topology given by the following convergence criterion:\\
For $i =1,2$, 
$v^n_i \to v_i$ in ${\mathcal S}_i$ as $n \to \infty$  if and only if
\begin{equation}\label{convergencecriterion}
\lim_{n \to \infty} \int_{\mathbb{R}^d} f(x) \int_{U_i} g(x, u_i) v^n_i(x)(du_i) dx \ = \  \int_{\mathbb{R}^d} f(x) \int_{U_i} g(x, u_i) v_i(x)(du_i) dx ,
\end{equation}
for all $f \in L^1(\mathbb{R}^d) \cap L^2(\mathbb{R}^d), \ g \in C_b(\mathbb{R}^d \times U_i)$; see
[\cite{AriBorkarGhosh}, p.57] for details. 

Now we define a class of strategies to be referred to as eventually stationary strategies denoted by $\hat{\mathcal S}_i,\; i=1,2$. Let $\Theta >0$. Let
\[
\hat{\mathcal S}_i \ = \ \{ \hat v_i : (0, \Theta) \times \mathbb{R}^d \to V_i \,  \; | \; \hat v_i \ {\rm is\ measurable} \}, \ 
i=1,2.
\]
 We consider the weak* topology  $L^\infty ((0, \Theta) \times {\mathcal M}_s(U_i)), $ on the space $\hat{\mathcal{S}}_i$, introduced by Warga \cite{Warga} for the topology of relaxed controls. Note that with the above topology, $\hat{\mathcal S}_i$ becomes a compact metrizable space with following convergence criterion:\\
For $i =1,2$, 
$\hat v^n_i \to \hat v_i$ in $\hat{\mathcal S}_i$ as $n \to \infty$  if and only if
\interdisplaylinepenalty=0
\begin{eqnarray}\label{convergencecriterion1}
&&\lim_{n \to \infty} \int_{(0, \Theta)}\int_{\mathbb{R}^d} f(\theta, x) \int_{U_i} g(\theta, x, u_i) \hat v^n_i(\theta, x)(du_i) dx d \theta  \\ \nonumber
& =&  \int_{(0, \Theta)}\int_{\mathbb{R}^d} f(\theta, x) \int_{U_i} g(\theta, x, u_i) \hat v_i(\theta, x)(du_i) 
dx d \theta  ,
\end{eqnarray}
for all $f \in L^1((0, \Theta) \times \mathbb{R}^d) \cap L^2((0, \Theta) \times \mathbb{R}^d), 
\ g \in C_b((0, \Theta) \times \mathbb{R}^d \times U_i)$. The Markov strategies associated with 
$\hat v_i \in \hat{\mathcal S}_i, i=1,2$, is given by $\hat v_i(\theta e^{-\alpha t}, X(t)), t \geq 0$, for each 
$\theta \in (0, \Theta)$ and $\alpha > 0$, where $X(t)$ is the solution of the s.d.e.
\[
d X(t) \ = \ b(X(t), \hat v_1(\theta e^{-\alpha t}, X(t)), \hat v_2(\theta e^{-\alpha t}, X(t))) dt + \sigma(X(t)) dW(t).
\]
By an abuse of notation, we represent the eventually stationary Markov strategies by elements of 
$\hat{\mathcal S}_i$, though each member in $\hat{\mathcal S}_i$ corresponds to a family
of Markov strategies indexed by $\theta$ and $\alpha$. Note that as $t \to \infty, \; e^{-\alpha t} \to 0$. Thus in the long run an element of $\hat{\mathcal S}_i$ ``eventually" becomes an element of $\mathcal S_i$ for a fixed $\theta$. Hence the terminology.

We consider two risk-sensitive cost criteria, discounted cost and ergodic cost criteria which we describe now.

\subsection{Discounted cost criterion} Let $\theta \in (0, \ \Theta)$ be the risk-aversion parameter. In the $\alpha$-discounted payoff criterion, $i$th
player chooses his strategy $v_i$ from the set of all Markov strategies ${\mathcal M}_i$
to minimize his risk-sensitive cost given by 
\begin{equation}\label{discountedcost}
\mathcal{J}^{v_1, v_2}_{\alpha, i}(\theta, x) \ := \ \dfrac{1}{\theta} \log  E^{v_1, v_2}_x
\Big[ e^{\theta \int^{\infty}_0 e^{-\alpha t} r_i(X(t), v_1(t,X(t)), v_2(t,X(t)) dt} \Big] , x \in \mathbb{R}^d, 
\end{equation}
where $\alpha > 0$ is the discount parameter,  $X(t)$ is the solution of the s.d.e. (\ref{statedynamics}) corresponding to $(v_1,v_2) \in  \mathcal{M}_1 \times \mathcal{M}_2 $ and $E^{v_1, v_2}_x$ denote the expectation with respect
to the law of the process (\ref{statedynamics}) corresponding to the Markov strategy pair $(v_1, v_2)$ with the initial condition $X(0) = x$.
\vspace{0.1in}
\begin{definition} A pair of strategies $(v^*_1, v^*_2) \in {\mathcal M}_1 \times {\mathcal M}_2$ is said to be a
Nash equilibrium among Markov strategies if 
\interdisplaylinepenalty=0
\begin{eqnarray*}
\mathcal{J}^{v^*_1, v^*_2}_{\alpha, 1}(\theta, x) & \leq & \mathcal{J}^{v_1, v^*_2}_{\alpha,1}(\theta, x) , \ \forall \ 
v_1 \in {\mathcal M}_1, \, x \in \mathbb{R}^d, \\
\mathcal{J}^{v^*_1, v^*_2}_{\alpha, 2}(\theta, x) & \leq & \mathcal{J}^{v^*_1, v_2}_{\alpha, 2}(\theta, x) , \ \forall \ 
v_2 \in {\mathcal M}_2, \, x \in \mathbb{R}^d.
\end{eqnarray*}
\end{definition}
Our main result for discounted risk-sensitive game is establishing the existence of a Nash equilibrium
among the class of eventually stationary strategies. 

\subsection{Ergodic cost criterion}
In this criterion player $i$ chooses his strategy $v_i \in {\mathcal M}_i$ so as to minimize his risk-sensitive
accumulated cost given by
\begin{equation}\label{risksensitivecost}
\rho^{v_1, v_2}_i (\theta, x) \ = \limsup_{T \to \infty} \frac{1}{\theta T} \log E^{v_1, v_2}_x
\Big[ e^{\theta \int^T_0 r_i(X(t), v_1(t,X(t)), v_2(t,X(t)) dt} \Big] , x \in \mathbb{R}^d.
\end{equation}
The definition of  Nash equilibrium is analogous.
We wish to establish the existence of a Nash equilibrium
among the class of  stationary strategies.
\vspace{0.1in}

For both cost criteria, we carry out  analysis by studying the corresponding system of coupled Hamilton-Jacobi-Bellman (HJB) equations. Note that if one of the players, say player $1$, is using a prescribed stationary strategy, then it is a (stochastic) optimal control problem for the other player (player $2$) which has been studied in \cite{Anup}, \cite{AnupBorkarSuresh}. The value function of this stochastic optimal control problem is the unique solution of the corresponding HJB equation. Then a stationary/Markov strategy associated with a minimizing selector of the appropriate Hamiltonian of the HJB equation yields an optimal control of the second player. Thus this stationary/Markov strategy is an optimal response of player $2$ given that player $1$ is employing a prescribed strategy. Therefore for a given pair of stationary/Markov strategies, we obtain a pair of optimal responses of the players via the corresponding HJB equations. Any fixed point of this map gives a Nash equilibrium. This leads us to study a coupled system of HJB equations for each criterion which we describe in forthcoming sections. To this end we first set up the frequently used notations.
\vspace{0.1in}

Denote $\displaystyle{ \sup_{v_1, v_2, x} |r_i(x, v_1, v_2)|}$ by $\|r_i\|_{\infty}, \ i =1,2.$
In general, for $\varphi \in C_b(\mathbb{R}^d),$ the space of all bounded, continuous functions,
we denote for each $B$, a Borel subset of $\mathbb{R}^d$,
\[
\| \varphi \|_{\infty, B} \ = \ \sup_{x \in B} | \varphi (x)|, \ \|\varphi\|_{\infty} = \sup_{x \in \mathbb{R}^d}
| \varphi(x)|.
\]

We define the weighted Sobolev spaces $W^{2, p, \lambda}(\mathbb{R}^d), 1 \leq p < \infty, \lambda >0$
 as follows:
 \interdisplaylinepenalty=0
\begin{eqnarray*}
W^{2, p, \lambda}(\mathbb{R}^d) & = & \{ \varphi : \mathbb{R}^d \to \mathbb{R} | \varphi 
\ {\rm is\ measurable\ and}\  \varphi   e(\lambda) , \frac{\partial \varphi}{\partial x_i}  e(\lambda), \\
&& \ \  \  \  \ \frac{\partial^2 \varphi}{\partial x_i \partial x_j} e(\lambda) \in L^p(\mathbb{R}^d), 
i, j = 1, \cdots , d\},
\end{eqnarray*}
where 
\[
e(\lambda)(x) \ = \ e^{ - \lambda \sqrt{ 1 + \|x\|^2}} , x \in \mathbb{R}^d.
\]
The $W^{2, p, \lambda}$-norm is defined as 
\[
\| \varphi \|^p_{2, p, \lambda} \ = \  \int | \varphi(x)|^p e(\lambda)(x) dx + 
\sum_i \int \Big| \frac{\partial \varphi(x)}{\partial x_i} \Big|^p e(\lambda)(x) dx + 
\sum_{i,j} \int \Big| \frac{\partial^2 \varphi(x)}{\partial x_i x_j} \Big|^pe(\lambda)(x)| dx.
\]
For a Banach space $\mathcal{X}$ with norm $\| \cdot \|_{\mathcal{X}}, \ 1 \leq p <\infty$, define
\[
L^p( t_0, T ; \mathcal{X}) \ = \ \{ \varphi : (t_0 , \ T) \to \mathcal{X} |  \varphi \ {\rm is\ Borel\ measurable\
and} \ \int^T_{t_0} \| \varphi (t) \|^p _{\mathcal{X}} \, dt  < \infty \}
\]
with the norm 
\[
\|\varphi\|_{p; \mathcal{X}} \ = \ \Big[ \int^T_{t_0} \|\varphi (t) \|^p_{\mathcal{X}} \, dt \Big]^{\frac{1}{p}}.
\]
The space ${\mathcal W}^{1,2,p}( (t_0, T) \times \mathbb{R}^d)), t_0 \geq 0$, denotes the set of all
$\varphi \in L^p ( t_0, T ; W^{2, p}(\mathbb{R}^d))$ such that 
$\frac{\partial \varphi}{\partial t} \in L^p ( (t_0, T; L^p(\mathbb{R}^d))$ with the norm
given by
\[
\|\varphi\|^p_{1,2,p; W^{2, p}(\mathbb{R}^d)} \ = \ 
\| \varphi \|^p_{p; W^{2,p}(\mathbb{R}^d)} + \| \frac{\partial \varphi}{\partial t} \|^p_{p; L^p(\mathbb{R}^d)}
, \ 1 \leq p < \infty.
\]
The corresponding weighted Sobolev spaces are defined by 
\[
{\mathcal W}^{1,2,p, \lambda}( (t_0, T) \times \mathbb{R}^d)) \ = \ 
\{ \varphi \in L^p( t_0, T; W^{2, p, \lambda}(\mathbb{R}^d) | \frac{\partial \varphi}{\partial t} \in 
L^p( t_0, T; L^{p, \lambda}(\mathbb{R}^d) \}
\]
with the norm given by 
\[
\|\varphi\|^p_{1,2, p, \lambda ; W^{2, p}(\mathbb{R}^d)} \ = \ 
\| \varphi \|^p_{p, ; W^{2,p, \lambda}(\mathbb{R}^d)} + 
\| \frac{\partial \varphi}{\partial t} \|^p_{p; L^{p, \lambda}(\mathbb{R}^d)}, \ 1 \leq p < \infty.
\]
It is easy to see that the above `regular' parabolic weighted Sobolev spaces can be isometrically 
identified with the  `usual' space-time weighted Sobolev spaces 
$W^{1,2,p, \lambda} ( (t_0, \ T) \times \mathbb{R}^d)$ defined by
\interdisplaylinepenalty=0
\begin{eqnarray*}
W^{1,2, p, \lambda}((t_0, \, T) \times \mathbb{R}^d) & = & 
\{ \varphi : (t_0, T) \times \mathbb{R}^d \to \mathbb{R} | \varphi 
\ {\rm is\ measurable\ and}\  \varphi   e(\lambda) , \frac{\partial \varphi}{\partial x_i}  e(\lambda), \\
&& \frac{\partial \varphi}{\partial t} e(\lambda),
 \frac{\partial^2 \varphi}{\partial x_i \partial x_j} e(\lambda) \in L^p((t_0, T) \times \mathbb{R}^d), 
i, j = 1, \cdots , d\}
\end{eqnarray*}
with norm 
\interdisplaylinepenalty=0
\begin{eqnarray*}
\| \varphi \|^p_{1, 2, p, \lambda} & = &  \int \int | \varphi(t, x)|^p e(\lambda)(x) dx dt  + 
\int \int \Big | \frac{\partial \varphi(t,x)}{\partial t}\Big |^p e(\lambda) dx dt \\
&& + \sum_i \int \int \Big| \frac{\partial \varphi(t,x)}{\partial x_i} \Big|^p e(\lambda)(x) dx  dt+ 
\sum_{ij} \int \int \Big| \frac{\partial^2 \varphi(t,x)}{\partial x_i x_j} \Big|^pe(\lambda)(x)| dx dt.
\end{eqnarray*}
Also the local Sobolev spaces $W^{1,2, p}((t_0 , T) \times \mathbb{R}^d)$ are defined by
\[
W^{1,2, p}_{{\rm loc}}(t_0, T) \times \mathbb{R}^d) \ = \ \{ \varphi : (t_0, T) \times \mathbb{R}^d \to
\mathbb{R}| \, \varphi \ {\rm is\ measurable\ and} \ \varphi \in W^{1,2,p}( (t_0, T) \times B_R), \ R > 0 \}.
\]
The norm $\| \cdot \|_{1,2,p; (t_0, T) \times B_R}$ is defined as 
\interdisplaylinepenalty=0
\begin{eqnarray*}
\|\varphi\|^p_{1,2,p; (t_0, T) \times B_R } & = &  \int^T_{t_0} \int_{B_R} | \varphi(t, x)|^p  dx dt  + 
\int^T_{t_0} \int_{B_R} \Big | \frac{\partial \varphi(t,x)}{\partial t}\Big |^p  dx dt \\
&& + \sum_i \int^T_{t_0} \int_{B_R} \Big| \frac{\partial \varphi(t,x)}{\partial x_i} \Big|^p  dx  dt+ 
\sum_{ij} \int^T_{t_0} \int_{B_R} \Big| \frac{\partial^2 \varphi(t,x)}{\partial x_i x_j} \Big|^p| dx dt,
\end{eqnarray*}
where $B_R$ denotes the open ball of radius $R$ with center $0$ in $\mathbb{R}^d$. 

\section{Analysis of Discounted Cost Criterion} In this section, we consider the discounted cost criterion for the stochastic differential game. We carry out our analysis for the $\alpha$-discounted cost criterion via the criterion 
\begin{equation}\label{discountedcost1}
J^{v_1, v_2}_{\alpha, i}(\theta, x) \ := \   E^{v_1, v_2}_x
\Big[ e^{\theta \int^{\infty}_0 e^{-\alpha t} r_i(X(t), v_1(t,X(t)), v_2(t,X(t)) dt} \Big] .
\end{equation}
Since logarithm is an increasing function, therefore any Nash equilibrium for the criterion (\ref{discountedcost}) is a Nash equilibrium for the above criterion.
The definition of  Nash equilibrium is analogous for the above criterion.

Let $\hat v_i \in \hat{\mathcal S}_i, \ i =1,2$. Corresponding to the cost criterion (\ref{discountedcost1}), the value functions are defined by
\interdisplaylinepenalty=0
\begin{eqnarray*}
\psi^{\hat v_2}_{\alpha, 1}(\theta, x) & = & \inf_{\tilde{v}_1 \in {\mathcal M}_1} J^{\tilde{v}_1, 
{\hat v_2}}_{\alpha, 1}(\theta, x)\\
& = & \inf_{\tilde{v}_1 \in {\mathcal M}_1} E^{\tilde{v}_1, {\hat v_2}}_x
\Big[ e^{\theta \int^{\infty}_0 e^{-\alpha t} r_i(X(t), \tilde{v}_1(t,X(t)), {\hat v_2}(\theta e^{-\alpha t},X(t)) dt}
 \Big], \\
\psi^{\hat v_1}_{\alpha, 2}(\theta, x) & = & \inf_{\tilde{v}_2 \in {\mathcal M}_2} J^{\hat v_1, \tilde{v}_2}_{\alpha, 2} (\theta, x)\\
 & = & \inf_{\tilde{v}_2 \in {\mathcal M}_2} E^{ {\hat v_1} , {\tilde{ v}_2}}_x
\Big[ e^{\theta \int^{\infty}_0 e^{-\alpha t} r_i(X(t), {\hat v_1} (\theta e^{-\alpha t},X(t)), \tilde{ v}_2(t,X(t)) dt}
 \Big], \ \theta \in (0, \ \Theta) , \ x \in \mathbb{R}^d.
\end{eqnarray*}
Now we prove that the above value functions are solutions of the corresponding HJB equations for discounted cost criterion. For a (heuristic) derivation of these HJB equations using multiplicative dynamic programming, we refer to \cite{MR}. We first prove the following.
\vspace{0.1in}
\begin{lemma}\label{appendixlemma-rep1} Assume (A1). Then for ${ \hat v_1} \in \hat{{\mathcal S}}_1 $
 and for each $\kappa \in (0, \ \Theta)$, the p.d.e.
 \interdisplaylinepenalty=0
\begin{eqnarray}\label{auxillarydiscountedhjb1}
\alpha \theta \frac{\partial \psi_{\kappa} }{\partial \theta} & = & 
\inf_{v_2 \in V_2} \Big[ \langle b(x, \hat v_1(\theta, x), v_2), \nabla_x \psi_{\kappa} \rangle + 
\theta r_2(x, \hat v_1(\theta, x), v_2) \psi_{\kappa}  \Big] \nonumber \\ \
&& + \frac{1}{2} \ {\rm trace}(a(x) \nabla^2_x \psi_{\kappa}) , \\ \nonumber 
\psi_{\kappa} (\kappa, x) & = & e^{\frac{\kappa \|r_2\|_{\infty}}{\alpha}}, \theta \in (\kappa, \ \Theta) , \ x \in \mathbb{R}^d, 
\end{eqnarray}
has a unique solution in $\bigcap_{p \geq d+1, \lambda > 0} W^{1,2,p, \lambda}((\kappa , \Theta) \times
\mathbb{R}^d)$  given by
\[
\psi_{\kappa}(\theta, x) \ = \ \inf_{v_2 \in {\mathcal M}_2}
E^{\hat v_1, v_2}_x \Big[e^{\frac{\kappa \|r_2\|_{\infty}}{\alpha}} e^{\theta \int^{T_\kappa}_0 e^{-\alpha t} 
r_2(X(t), \hat v_1(\theta e^{-\alpha t},X(t)), v_2(t, X(t)))dt}\Big],
\]
where $\displaystyle{T_{\kappa} = \frac{\log(\frac{\theta}{\kappa})}{\alpha}}$. 

Similarly, for ${ \hat v_2} \in \hat{{\mathcal S}}_2 $, the p.d.e.
\interdisplaylinepenalty=0
\begin{eqnarray}\label{auxillarydiscountedhjb12}
\alpha \theta \frac{\partial \phi_{\kappa} }{\partial \theta} & = & 
\inf_{v_1 \in V_1} \Big[ \langle b(x,  v_1, \hat v_2(\theta, x) ), \nabla_x \phi_{\kappa} \rangle + 
\theta r_1(x, v_1, \hat v_2(\theta, x) ) \phi_{\kappa}  \Big] \nonumber \\ \
&& + \frac{1}{2} \ {\rm trace}(a(x) \nabla^2_x \phi_{\kappa}) , \\ \nonumber 
\phi_{\kappa} (\kappa, x) & = & e^{\frac{\kappa \|r_1\|_{\infty}}{\alpha}}, \theta \in (\kappa, \ \Theta) , \ x \in \mathbb{R}^d, 
\end{eqnarray}
has a unique solution in $\bigcap_{p \geq d+1, \lambda > 0} W^{1,2,p, \lambda}((\kappa , \Theta) \times
\mathbb{R}^d)$  given by
\[
\phi_{\kappa}(\theta, x) \ = \ \inf_{v_1 \in {\mathcal M}_1}
E^{v_1, \hat v_2}_x \Big[e^{\frac{\kappa \|r_2\|_{\infty}}{\alpha}} e^{\theta \int^{T_\kappa}_0 e^{-\alpha t} 
r_1(X(t), v_1( t,X(t)), \hat v_2(\theta e^{- \alpha t}, X(t)))dt}\Big].
\]
\end{lemma}
\begin{proof} 
From [\cite{BensoussanLions}, Theorem 3.3, p.235-236], it follows that (\ref{auxillarydiscountedhjb1})
has a unique solution in 
${\mathcal W}^{1,2,p, \lambda}((\kappa, \Theta) \times \mathbb{R}^d) ,  p \geq 2, \lambda > 0$. 
Since  
${\mathcal W}^{1,2,p}((\kappa, \Theta) \times \mathbb{R}^d) $ is isometric to  
$W^{1,2, p, \lambda}((\kappa , \Theta) \times \mathbb{R}^d)$, it follows that the p.d.e. 
(\ref{auxillarydiscountedhjb1}) has a unique solution in 
$\bigcap_{p \geq 2, \lambda > 0} W^{1,2,p, \lambda}((\kappa , \Theta) \times \mathbb{R}^d)$.

Fix $p \geq d+1$ and $\lambda > 0$. Choose a sequence 
$\psi^n \in C^{\infty}_0 ( (\kappa, \Theta) \times \mathbb{R}^d)$, the space of all 
$C^{\infty} ((\kappa, \Theta) \times \mathbb{R}^d)$ functions which are compactly supported, such that
$\psi^n \to \psi_{\kappa}$ in $W^{1,2,p, \lambda}((\kappa, \Theta) \times \mathbb{R}^d)$.

Now using It$\hat{\rm o}$-Dynkin  formula to $\psi^n$ 
$(\theta(t), X(t)), t \geq 0,$ where  $ \theta(t) =  \theta e^{-\alpha t}$ and $X(t)$ is the process
(\ref{statedynamics}) corresponding to the Markov strategy pair $(\hat v_1, v_2), v_2 \in {\mathcal M}_2$ 
with the initial condition $X(0) = x$,  , we obtain
\interdisplaylinepenalty=0
\begin{eqnarray*}\label{eq0appendix}
\psi^n(\theta, x) & = &   E^{\hat v_1, v_2}_x \Big[ e^{\frac{\kappa \|r_2\|_{\infty}}{\alpha}} 
 e^{\theta\int^{T_{\kappa} \wedge \tau_R}_0 e^{- \alpha s}
 r_2(X(s), \hat v_1(\theta(s),X(s)), v_2(s, X(s))) ds }  \Big]\nonumber \\
&& - E^{\hat v_1, v_2}_x \Big[ \int^{T_{\kappa} \wedge \tau_R}_0 
e^{\theta \int^t_0 e^{-\alpha s} r_2(X(s), \hat v_1(\theta(s),X(s)), v_2(s, X(s))) ds } \Big[ 
- \alpha \theta(t) \frac{\partial \psi^n(\theta(t), X(t))}{\partial \theta}    \nonumber \\
&&  + \langle b(X(t), \hat v_1(\theta(t),X(t)), v_2(t, X(t))) , \nabla_x \psi^n(\theta(t), X(t)) \rangle   \nonumber \\
&&  +  \theta(t) r_2(X(t),\hat v_1(\theta(t),X(t)), v_2(t, X(t))) \psi^n(\theta(t), X(t))   \nonumber \\
&&  + \frac{1}{2} {\rm trace}(a(X(t))  \nabla^2_x \psi^n(\theta(t), X(t))) \Big] dt \Big],\nonumber  \\ 
\end{eqnarray*}
where $$\tau_R \ = \ \inf \{ t \geq 0 | \|X(t) \| > R \} , R > 0.$$
Now by letting $n \to \infty$, we obtain
\interdisplaylinepenalty=0
\begin{eqnarray*}
\psi_{\kappa}(\theta, x) & = &   E^{\hat v_1, v_2}_x \Big[ e^{\frac{\kappa \|r_2\|_{\infty}}{\alpha}} 
 e^{\theta\int^{T_{\kappa} \wedge \tau_R}_0 e^{- \alpha s}
 r_2(X(s),\hat v_1(\theta(s),X(s)), v_2(s, X(s))) ds }  \Big]\nonumber \\
&& - E^{v_1, v_2}_x \Big[ \int^{T_{\kappa} \wedge \tau_R}_0 
e^{\theta \int^t_0 e^{-\alpha s} r_2(X(s), \hat v_1(\theta(s), X(s)), v_2(s, X(s))) ds } \Big[ 
- \alpha \theta(t) \frac{\partial \psi_{\kappa}(\theta(t), X(t))}{\partial \theta}   \\ \nonumber 
&&  + \langle b(X(t),\hat v_1(\theta(t),X(t)), v_2(t, X(t))) , \nabla_x \psi_{\kappa}(\theta(t), X(t)) 
\rangle  \\ \nonumber
&&  +  \theta(t) r_2(X(t),\hat v_1(\theta(t),X(t)), v_2(t, X(t))) \psi^n(\theta(t), X(t))  \\ \nonumber 
&&  + \frac{1}{2} {\rm trace}(a(X(t))  \nabla^2_x \psi_{\kappa}(\theta(t), X(t))) \Big] dt \Big]. \\ \nonumber 
\end{eqnarray*}
 Hence it follows that
\[
\psi_{\kappa}(\theta, x) \ \leq \ E^{\hat v_1, v_2}_x \Big[ e^{\frac{\kappa \|r_2\|_{\infty}}{\alpha}} 
e^{\theta \int^{T_{\kappa} \wedge \tau_R}_0 e^{-\alpha s} r_2(X(t), \hat v_1(\theta(t), X(t)), v_2(t, X(t))) dt}
 \Big], \ \forall \  v_2 \in {\mathcal M}_2.
\]
Now by invoking dominated convergence theorem for letting $R \to \infty$ above, we obtain 
\[
\psi_{\kappa}(\theta, x) \ \leq \ E^{\hat v_1, v_2}_x \Big[ e^{\frac{\kappa \|r_2\|_{\infty}}{\alpha}} 
e^{\theta \int^{T_{\kappa}}_0 e^{-\alpha s} r_2(X(t), \hat v_1(\theta(t), X(t)), v_2(t, X(t))) dt}
 \Big], \ \forall \  v_2 \in {\mathcal M}_2.
\]
Let $\hat v^*_2(\theta, x)$ denote a (measurable) minimizing selector in
\[
\inf_{v_2 \in V_2} \Big[ \langle b(x, \hat v_1(\theta, x), v_2), \nabla_x \psi_{\kappa}(\theta, x) \rangle + 
\theta r_2(x, \hat v_1(\theta, x), v_2) \psi_{\kappa}(\theta, x)  \Big] .
\]
That is 
\interdisplaylinepenalty=0
\begin{eqnarray*}
&& \inf_{v_2 \in V_2} \Big[ \langle b(x, \hat v_1(\theta, x), v_2), \nabla_x \psi_{\kappa}(\theta, x) \rangle + 
\theta r_2(x, \hat v_1(\theta, x), v_2) \psi_{\kappa}(\theta, x)  \Big]
 \\ \nonumber
&& =  \langle b(x, \hat v_1(\theta, x), \hat v^*_2(\theta, x)), \nabla_x \psi_{\kappa}(\theta, x) \rangle + 
\theta r_2(x, \hat v_1(\theta, x), \hat v^*_2(\theta, x)) \psi_{\kappa}(\theta, x). 
\end{eqnarray*}
The existence of such a $\hat v^*_2(\theta, x)$ is ensured by \cite{Benes}.
Repeating the above argument replacing $v_2$ with $\hat v^*_2 \in \hat {\mathcal S}_2$, we obtain
\[
\psi_{\kappa}(\theta, x) \ =  \ E^{\hat v_1, \hat v^*_2}_x \Big[ e^{\frac{\kappa \|r_2\|_{\infty}}{\alpha}} 
e^{\theta \int^{T_{\kappa}}_0 e^{-\alpha s} r_2(X(t), \hat v_1(\theta(t),X(t)), \hat v^*_2(\theta(t), X(t))) dt} \Big] .
\]
Thus 
\begin{equation}\label{eq1appendix}
\psi_{\kappa}(\theta, x) \ = \ \inf_{v_2 \in {\mathcal M}_2} 
E^{\hat v_1, v_2}_x \Big[ e^{\frac{\kappa \|r_2\|_{\infty}}{\alpha}} 
e^{\theta \int^{T_{\kappa}}_0 e^{-\alpha s} r_2(X(t), \hat v_1(\theta(t),X(t)), v_2(t, X(t))) dt} \Big].
\end{equation}
This completes the proof of the first part. The proof of the second part follows by a symmetric argument.
\end{proof}
\vspace{0.1in}

Next we take limit $\kappa \to 0$ and show that value function satisfies the limiting equation. In particular we prove the following theorem.
\vspace{0.1in}
\begin{theorem}\label{pdeexistence} Assume (A1). (i)
For each $\hat v_1 \in \hat{\mathcal S}_1$, the p.d.e. 
\interdisplaylinepenalty=0
\begin{eqnarray}\label{discountedhjb1}
\alpha \theta \frac{\partial \psi}{\partial \theta} & = & 
\inf_{v_2 \in V_2} \Big[ \langle b(x, \hat v_1(\theta, x), v_2), \nabla_x \psi \rangle 
+ \theta r_2(x, \hat v_1(\theta, x), v_2) \psi  \Big] \nonumber \\ 
&& + \frac{1}{2} \ {\rm trace}(a(x) \nabla^2_x \psi) , \\ \nonumber 
\psi (0, x) & = & 1, \theta \in (0, \ \Theta) , \ x \in \mathbb{R}^d
\end{eqnarray}
has a unique solution in 
$\bigcap_{p \geq d+1} W^{1, 2, p }_{{\rm loc}} ( (0, \ \Theta) \times \mathbb{R}^d) 
\cap C_b((0, \ \Theta) \times \mathbb{R}^d)$
given by 
$$\inf_{v_2 \in {\mathcal M}_2}
E^{v_1, v_2}_x \Big[e^{\theta \int^{T_\kappa}_0e^{-\alpha t} 
r_2(X(t), \hat v_1(\theta e^{-\alpha t}, X(t)), v_2(t, X(t)))dt}\Big] \ := \ \psi^{v_1}_{\alpha, 2} (\theta, x).$$

(ii) Similarly, for each $\hat v_2 \in \hat{\mathcal S}_2$, the p.d.e. 
\interdisplaylinepenalty=0
\begin{eqnarray}\label{discountedhjb2}
\alpha \theta \frac{\partial \phi}{\partial \theta} & = & 
\inf_{v_1 \in V_1} \Big[ \langle b(x, v_1, \hat v_2(\theta, x)), \nabla_x \phi \rangle 
+ \theta r_1(x, v_1, \hat v_2(\theta, x)) \phi  \Big] \nonumber \\ 
&& + \frac{1}{2} \ {\rm trace}(a(x) \nabla^2_x \phi) , \\ \nonumber 
\phi (0, x) & = & 1, \theta \in (0, \ \Theta) , \ x \in \mathbb{R}^d
\end{eqnarray}
has a unique solution in 
$\bigcap_{p \geq d+1} W^{1, 2, p}_{{\rm loc}} ( (0, \ \Theta) \times \mathbb{R}^d) 
\cap C_b((0, \ \Theta) \times \mathbb{R}^d)$
 given by $$ \inf_{v_1 \in {\mathcal M}_1}
E^{v_1, v_2}_x \Big[e^{\theta \int^{T_\kappa}_0e^{-\alpha t} 
r_1(X(t), \hat v_1(\theta e^{-\alpha t}, X(t)), v_2(t, X(t)))dt}\Big] \ := \ \psi^{v_2}_{\alpha, 1} (\theta, x). $$
\end{theorem}
\begin{proof} We  prove (i). The proof of (ii) is analogous. 
From Lemma \ref{appendixlemma-rep1}, it follows that (\ref{auxillarydiscountedhjb1})
has a unique solution in 
$W^{1,2,p, \lambda}((\kappa, \Theta) \times \mathbb{R}^d) 
\cap C_b((\kappa, \, \Theta) \times \mathbb{R}^d), \ p \geq2, \lambda > 0$ and 
is given by
\begin{equation}\label{eq1section3}
\psi_{\kappa}(\theta, x) \ = \ \inf_{v_2 \in {\mathcal M}_2}
E^{v_1, v_2}_x \Big[e^{\frac{\kappa \|r_2\|_{\infty}}{\alpha}} e^{\theta \int^{T_\kappa}_0e^{-\alpha t} 
r_2(X(t), \hat v_1(\theta e^{-\alpha t}, X(t)), v_2(t, X(t)))dt}\Big].
\end{equation}
From (\ref{eq1section3}), it follows that
\begin{equation}\label{eq2section3}
\|\psi_{\kappa} \|_{\infty} \ \leq \ e^{\frac{\theta \|r_2 \|_{\infty} }{\alpha}}.
\end{equation}
Using similar arguments as in the proof of [\cite{AnupBorkarSuresh}, Theorem 3.1], it follows that
\begin{equation}\label{eq3section3}
\|\frac{\partial \psi_{\kappa}}{\partial \theta} \|_{\infty} \ \leq \ 3 
e^{\frac{(\theta + 2) \|r_2\|_{\infty}}{\alpha}} \frac{\|r_2\|_{\infty}}{\alpha}.
\end{equation} 
For $\lambda > 0$ fixed, rewrite (\ref{auxillarydiscountedhjb1}) as follows.
\begin{equation}\label{auxillaryellipticpde1}
\lambda \psi_{\kappa} = \ \inf_{v_2 \in V_2} \Big[ \langle b(x, \hat v_1(\theta, x), v_2) , 
\nabla_x \psi_{\kappa} \rangle
+ r_{\kappa}(\theta, x, v_2) \Big] + \frac{1}{2} {\rm trace}(a(x) \nabla^2_x \psi_{\kappa}),
\end{equation}
where
\[
r_{\kappa}(\theta, x, v_2) \ = \ \theta r_2(x, \hat v_1(\theta, x), v_2) \psi_{\kappa}(\theta, x) - \alpha \theta 
\frac{\partial \psi_{\kappa}}{\partial \theta} + \lambda \psi_{\kappa}(\theta, x).
\]
Choose $\lambda \geq \lambda_0 > 0$ is such that $r_k \geq 0$. 
From [\cite{BensoussanLions}, Lemma 1.5, p.209], for $\lambda  > 0$ large enough, 
say $\lambda \geq \lambda_1 $ for some large $\lambda_1 \geq \lambda_0$, for each $ \theta > \kappa$,
$\psi_{\kappa}(\theta, \cdot) \in \bigcap_{p \geq 2, \lambda > \lambda_1} W^{2, p, \lambda}(\mathbb{R}^d) 
\cap C_b(\mathbb{R}^d)$ is the unique solution to (\ref{auxillaryellipticpde1}). 
Let $\hat v_2(\theta, x)$ be a minimizing selector in (\ref{auxillaryellipticpde1}). Then 
(\ref{auxillaryellipticpde1}) can be rewritten as 
\[
\langle b(x, \hat v_1(\theta, x), \hat v_2(\theta, x)) , \nabla_x \psi_{\kappa} \rangle 
+ \frac{1}{2} {\rm trace}(a(x) \nabla^2_x
 \psi_{\kappa})  = \lambda \psi_{\kappa}  - r_k(x, \hat v_2(\theta, x)).
\]
The r.h.s. of the above p.d.e. is uniformly bounded in $\kappa > 0$. Hence by using similar arguments as in
 [\cite{Borkar}, p.158], it follows that 
\begin{equation}\label{eq4section3}
\|\psi_{\kappa}(\theta, \cdot) \|_{2, p; B_R}  \leq K,
\end{equation}
where the constant $K >0$ is independent of $\kappa > 0, \theta \in (0, \ \Theta)$. 

Define $\Bar{\psi}_{\kappa}$ as follows.
\[
\Bar{\psi}_{\kappa}(\theta, x) \ = \ 
\left\{
\begin{array}{lll}
\psi_{\kappa}(\theta, x) & {\rm if}& \theta > \kappa\\
e^{\frac{\kappa \|r_2\|_{\infty}}{\alpha}} &{\rm if}& \theta \leq \kappa.
\end{array}
\right.
\]
One can see that $\Bar{\psi}_{\kappa} \in  W^{1,2,p}((0, \Theta) \times B_R)$
for all $p \geq 2, R > 0$.  Also from (\ref{eq3section3}) and (\ref{eq4section3}), it follows that
\begin{equation}\label{eq5section3}
\sup_{\kappa > 0} \|\Bar{\psi}_{\kappa}\|_{1,2,p; B_R} < \infty,  \ R > 0.
\end{equation}
Note that $W^{1,2,p}( (0, \Theta) \times B_R)$ is a reflexive Banach space. Hence
using Banach-Alaoglu theorem, it follows that $\{\psi_{\kappa} | \kappa > 0\}$ is weakly
compact in $W^{1, 2, p}( (0, \ \Theta) \times B_R), \, R > 0$.  Hence by a diagonalization 
procedure, there exists 
$\psi \in W^{1,2, p}_{{\rm loc}} ( (0, \Theta) \times \mathbb{R}^d)$ and a sequence, say $\psi_{\kappa_n}$,
such that $\psi_{\kappa_n}  \to \psi$ weakly in $W^{1,2, p, \lambda}((0, \Theta) \times B_R), R > 0$.
Now by a standard approximation argument, by letting $\kappa_n \to 0$ in 
(\ref{auxillarydiscountedhjb1}), it follows that 
$\psi \in W^{1,2,p}_{{\rm loc}}((0, \Theta) \times \mathbb{R}^d)$
is a solution in the sense of distributions  to 
\interdisplaylinepenalty=0
\begin{eqnarray*}
\alpha \theta \frac{\partial \psi }{\partial \theta} & = & 
\inf_{v_2 \in V_2} \Big[ \langle b(x, \hat v_1(\theta, x), v_2), \nabla_x \psi \rangle + 
\theta r_2(x, \hat v_1(\theta, x), v_2) \psi  \Big] \\ 
&& + \frac{1}{2} \ {\rm trace}(a(x) \nabla^2_x \psi) .
\end{eqnarray*}
Moreover $\psi(0, x) = 1$. Hence $\psi \in W^{1,2,p}_{{\rm loc}}((0, \Theta) \times \mathbb{R}^d)$
is a solution to (\ref{discountedhjb1}).

For $p \geq d+1$, using It$\hat{\rm o}$'s formula 
as in the proof of Lemma \ref{appendixlemma-rep1}, it follows that 
\[
\psi (\theta, x) \ = \ \inf_{v_2 \in {\mathcal M}_2}
E^{v_1, v_2}_x \Big[e^{\theta \int^{\infty}_0e^{-\alpha t} r_2(X(t), \hat v_1(\theta e^{-\alpha t}, X(t)), 
v_2(t, X(t)))dt}\Big](:=\psi^{\hat v_1}_{\alpha,2} (\theta, x)).
\]
Hence $\psi \in \bigcap_{p \geq d+1} W^{1,2,p}_{{\rm loc}}((0, \Theta) \times \mathbb{R}^d)\cap C_b((0, \ \Theta) \times \mathbb{R}^d)$
is the unique solution to (\ref{discountedhjb1}). 
\end{proof}

\vspace{0.1in}
To continue our analysis we prove some estimates needed later.
\vspace{0.1in}
\begin{lemma} \label{lemma-estimates1} Assume (A1). 
(i) For $\theta \in (0, \Theta)$ and $\alpha > 0$ and $\hat v_i \in \hat{\mathcal S}_i, i =1,2$, we have
\interdisplaylinepenalty=0
\begin{eqnarray}\label{eq1estimates}
1 \leq \max \{ \psi^{\hat v_2}_{\alpha, 1} (\theta, x) , \psi^{\hat v_1}_{\alpha, 2} (\theta, x) \} 
\leq \max_{i}e^{\frac{\theta \|r_i\|_{\infty}}{\alpha}}, & \\ \nonumber 
\max \{ \|\frac{\partial \psi^{\hat v_2}_{\alpha, 1}}{\partial \theta} \|_{\infty}, 
\|\frac{\partial \psi^{\hat v_1}_{\alpha, 2}}{\partial \theta} \|_{\infty} \} 
 \ \leq \ \max_{i}\frac{\|r_i\|_{\infty}}{\alpha} e^{\frac{\Theta \| r_i\|_{\infty}}{\alpha}} . & 
\end{eqnarray}

(ii) For each $R > 0$, and $\hat v_i \in \hat{\mathcal S}_i, i =1,2$, we have
\begin{equation}\label{eq2estimates1}
\sup_{ \hat v_1 \in \hat {\mathcal S}_1 } \|\psi^{\hat v_1}_{\alpha, 2} \|_{1,2, p;  (0, \Theta) \times B_R} < \infty, \ 
\sup_{ \hat v_2 \in \hat {\mathcal S}_2 } \|\psi^{\hat v_2}_{\alpha, 1} \|_{1,2, p;  (0, \Theta) \times B_R} < \infty . 
\end{equation}

\end{lemma}
\begin{proof}  The proof of (i) follows as in [\cite{AnupBorkarSuresh}, Theorem 3.1]. 

\vspace{.1in}

Let $\tilde{v}_2(\theta, x)$ be a minimizing selector in (\ref{discountedhjb1}).  
Then (\ref{discountedhjb1}) can be rewritten as 
\interdisplaylinepenalty=0
\begin{eqnarray*}
&&\langle b(x, \hat v_1(\theta, x), \tilde{v}_2(\theta, x)) , \nabla_x \psi^{\hat v_1}_{\alpha, 2} \rangle 
+ \frac{1}{2} \ {\rm trace}
(a(x) \nabla^2_x \psi^{\hat v_1}_{\alpha, 2} ) \\ &= & \alpha \theta 
\frac{\partial \psi^{\hat v_1}_{\alpha, 2}}{\partial \theta} 
- \theta r_2(x, \hat v_1(\theta, x), \tilde{v}_2(\theta , x)) \psi^{\hat v_1}_{\alpha, 2} .
\end{eqnarray*}
From (i), it follows that r.h.s. above is uniformly bounded in $\theta \in (0, \Theta), \hat v_1 \in \hat{\mathcal S}_1$. 
Hence using arguments as in [\cite{Borkar}, p.158], it follows that 
\begin{equation}\label{eq3estimates1}
\sup_{\theta, \hat v_1 } \|\psi^{\hat v_1}_{\alpha, 2} \|_{2, p ; B_R }  < \infty.
\end{equation}
Hence (ii) follows. 
\end{proof}

\vspace{.1in}

\begin{lemma}\label{lemma-continuousmap} Assume (A1). The maps $ \hat v_1 \mapsto \psi^{\hat v_1}_{\alpha, 2} $ 
from $\hat{\mathcal S}_1 \to C^{0, 1}((0, \Theta)  \times \mathbb{R}^d)$  
and $ \hat v_2 \mapsto \psi^{\hat v_2}_{\alpha, 1} $
 from $\hat{\mathcal S}_2  \to C^{0, 1}((0, \Theta)  \times \mathbb{R}^d)$ are continuous.
\end{lemma}
\begin{proof}  Let $\hat v^m_1 \to \hat v_1 $ in $\hat{\mathcal S}_1$. Then using Lemma \ref{lemma-estimates1} (ii),
we have for each $R > 0$,
\begin{equation}\label{eq1continuousmap}
\sup_{m  \geq 1} \|\psi^{\hat v^m_1}_{\alpha, 2} \|_{1,2,p; (0, \Theta) \times B_R} < \infty.
\end{equation}
Now by compact embedding theorem  [see \cite{Adams}, Chapter 6], Banach-Alaoglu theorem, 
and a standard diagonalization argument,
there exist $\psi \in W^{1,2,p}_{\rm loc}((0, \Theta ) \times \mathbb{R}^d) \cap C^{0, 1} ( (0, \ \Theta ) \times \mathbb{R}^d)$
 such that along a subsequence
\begin{equation}\label{eq2continuousmap}
\psi^{\hat v^m_1}_{\alpha, 2}  \to \psi \ {\rm in} \ C^{0, 1} ((0, \Theta) \times B_R),  \ R > 0.
\end{equation}
For $\varphi \in C^{\infty}_0 ( (0, \Theta) \times \mathbb{R}^d)$, from (\ref{discountedhjb1}) we have
\interdisplaylinepenalty=0
\begin{eqnarray*}
&&\int_0^\Theta \int_{\mathbb{R}^d} \alpha \theta \frac{\partial \psi^{\hat v^m_1}_{\alpha ,2}}{\partial \theta} \varphi d\theta dx \\ & = & \int_0^\Theta \int_{\mathbb{R}^d}
\inf_{v_2 \in V_2} \Big[ \langle b(x, \hat v^m_1(\theta, x), v_2), \nabla_x \psi^{\hat v^m_1}_{\alpha, 2}  \rangle 
+ \theta r_2(x, \hat v^m_1(\theta, x), v_2) \psi^{\hat v^m_1}_{\alpha, 2} \Big] \varphi d \theta d x \\
 &+& \frac{1}{2} \ \int_0^\Theta \int_{\mathbb{R}^d} {\rm trace}(a(x) \nabla^2_x \psi^{\hat v^m_1}_{\alpha, 2} ) \varphi d \theta dx.
\end{eqnarray*}
Now using appropriate integration by parts formulae for the second term on the r.h.s. and 
the term on the l.h.s above and then using (\ref{eq2continuousmap}),
it follows that $\psi$ is a solution in the sense of distribution to (\ref{discountedhjb1}). Since $\psi(0, x) =1$, it is a solution in $\bigcap_{ p \geq d+1} W^{1,2,p}_{{\rm loc}} (0, \Theta)
\times \mathbb{R}^d)\cap C_b ((0, \Theta)\times \mathbb{R}^d)$ of the p.d.e. (\ref{discountedhjb1}). Hence by Theorem \ref{pdeexistence},
it follows that $\psi = \psi^{\hat v_1}_{\alpha, 1}$. This proves the continuity of the first map. The proof of the
continuity of  the second map follows by a  symmetric argument. 
\end{proof}

\vspace{.1in}
For $(\hat v_1,\hat v_2) \in \hat{S}_1 \times \hat{S}_2$, define
\begin{equation}\label{optimalresponses}
H(\hat v_1, \hat v_2) \ = \ H_1(\hat v_2) \times H_2(\hat v_1),
\end{equation}
where
\begin{equation*}\label{optimalrespeonse1}
H_1(\hat v_2) \ = \ \Big\{\hat v^*_1 \in \hat{\mathcal S}_1 | F_1 (x, \hat v^*_1(\theta, x), \hat v_2(\theta, x)) = 
\inf_{v_1 \in V_1} F_1(x,  v_1,\hat v_2(\theta, x)) \ {\rm a.e.} \ \theta, x \Big\},
\end{equation*}
\begin{equation*}
F_1(x, v_1, \hat v_2(\theta, x) ) \ = \ \langle b(x, v_1, \hat v_2(\theta, x)) , \nabla_x \psi^{\hat v_2}_{\alpha, 1} \rangle
+ \theta r_1(x, v_1, \hat v_2(\theta, x)), 
\end{equation*}
$\theta , x \in (0, \Theta) \times \mathbb{R}^d, \ v_1 \in V_1, \hat v_2 \in \hat{\mathcal S}_2,$
\begin{equation*}\label{optimalrespeonse2}
H_2(\hat v_1) \ = \ \Big\{\hat v^*_2 \in \hat{\mathcal S}_2 | F_2 (x, \hat v_1(\theta, x), \hat v^*_2(\theta, x)) = 
\inf_{v_2 \in V_2} F_1(x, \hat v_1(\theta, x), v_2) \ {\rm a.e.} \ \theta, x \Big\},
\end{equation*}
\begin{equation*}
F_2(x, \hat v_1(\theta, x), v_2 ) \ = \ \langle b(x,\hat v_1(\theta, x) , v_2) , \nabla_x \psi^{\hat v_1}_{\alpha, 2} \rangle
+ \theta r_2(x,\hat v_1(\theta, x), v_2), 
\end{equation*}
$\theta , x \in (0, \Theta) \times \mathbb{R}^d, \ v_2 \in V_2, \hat v_1 \in \hat{\mathcal S}_1,$

\vspace{.1in}

Using linear structure of $F_i$ and the compactness
of $\hat{\mathcal S}_i, i =1,2$, it is easy to see that $H(\hat v_1, \hat v_2)$ is nonempty, convex and compact subset of 
$\hat{\mathcal S}_1 \times \hat{\mathcal S}_2$.

To prove the existence of a Nash equilibrium, we prove the upper semi-continuity (u.s.c.) of the map
$(\hat v_1, \hat v_2) \mapsto H(\hat v_1, \hat v_2) $ from $\hat{\mathcal S}_1 \times \hat{\mathcal S}_2 \to 
2^{\hat{\mathcal S}_1} \times 2^{\hat{\mathcal S}_2}$.  To establish the u.s.c. we need some  
additive   structure on the drift of the state dynamics and the cost function (ADAC) given as follows. \\

\noindent {\bf (A2)} We assume that $\Bar{b} : \mathbb{R}^d \times U_1 \times U_2 \to \mathbb{R}^d$
and $\Bar{r}_i : \mathbb{R}^d \times U_1 \times U_2 \to [0, \ \infty), \ i =1,2$ take the additive
structure given by
\interdisplaylinepenalty=0
\begin{eqnarray*}
\Bar{b}(x, u_1, u_2) & = & \Bar{b}_1(x, u_1) + \Bar{b}_2(x, u_2), \\
\Bar{r}_i (x, u_1, u_2) & = & \Bar{r}_{i1}(x, u_1) + \Bar{r}_{i2}(x, u_2), x \in \mathbb{R}^d, u_1 \in U_1,
u_2 \in U_2, i =1,2.
\end{eqnarray*}

The conditions in (A1)(i) are assumed for $\bar{b}_i, \bar{r}_{i1}, \bar{r}_{i2}, i=1,2 $.

\vspace{.1in}
\begin{lemma}\label{usc1} Assume (A1) and (A2). Then the map 
$(\hat v_1, \hat v_2) \mapsto H(\hat v_1, \hat v_2) $ from $\hat{\mathcal S}_1 \times \hat{\mathcal S}_2 \to 
2^{\hat{\mathcal S}_1} \times 2^{\hat{\mathcal S}_2}$ is u.s.c.
\end{lemma}
\begin{proof} Let $(\hat v^m_1, \hat v^m_2) \to (\hat v_1, \hat v_2) \in \hat{\mathcal S}_1 \times \hat{\mathcal S}_2$
and $\hat {\Bar{v}}^m_1 \in H_1(\hat v^m_2), m \geq 1$.  Since $\hat{\mathcal S}_1$ is compact, it follows that
there exists a  limit point, say $\hat{\Bar{v}}_1$ for $\{ \hat{\Bar{v}}^m_1\}$. By an abuse of notation, we 
write $\hat{\Bar{v}}^m_1 \to \hat{\Bar{v}}_1$ in $\hat{\mathcal S}_1$. 
Now using (A2), Lemma \ref{lemma-continuousmap} and the topology of $\hat{\mathcal S}_i, i =1,2$, it follows that $$\langle b(x, \hat{\Bar{v}}^m_1(\theta, x), \hat v^m_2(\theta, x)),\nabla_x \psi^{\hat v^m_2}_{\alpha, 1} \rangle
+ \theta r_1(x, \hat{\Bar{v}}^m_1(\theta, x),\hat v^m_2(\theta, x))$$
converges weakly in $ L^{2}_{{\rm loc}} (0, \Theta)
\times \mathbb{R}^d)$ to
$$\langle b(x, \hat{\Bar{v}}_1(\theta, x), \hat v_2(\theta, x)),\nabla_x \psi^{\hat v_2}_{\alpha, 1} \rangle
+ \theta r_1(x, \hat{\Bar{v}}_1(\theta, x), \hat v_2(\theta, x)).$$
Therefore by Banach-Saks theorem any sequence of convex combination of the former converges strongly in $ L^{2}_{{\rm loc}} (0, \Theta)
\times \mathbb{R}^d)$  to the latter. Hence along a suitable subsequence 
\begin{equation}\label{eq1usc}
\lim_{m \to \infty} F(x, \hat{ \Bar{v}}^m_1(\theta, x), \hat v^m_2(\theta, x)) 
= F(x, \hat{\Bar{v}}_1(\theta, x), \hat v_2(\theta, x)), \ {\rm a.e.\ in}\ \theta, x.
\end{equation}
(In above we denote the convex combination coming from the Banach-Saks theorem by 
$F(x, \hat{\Bar{v}}^m_1(\theta, x), \hat{v}^m_2(\theta, x))$ itself by an abuse of notation.)
Now fix $\hat {\tilde{v}}_1 \in \hat S_1$ and use analogous arguments above to conclude that 
\begin{equation}\label{eq2usc}
\lim_{m \to \infty} F(x,  \hat {\tilde{v}}_1(\theta, x), \hat v^m_2(\theta, x)) 
= F(x, \hat {\tilde{v}}_1(\theta, x), \hat v_2(\theta, x)), \ {\rm a.e.\ in}\ \theta, x.
\end{equation} 
Since $\hat{\Bar{v}}^m_1 \in H_1(\hat v^m_2),$ we have
\[
F(x,  \hat {\tilde{v}}_1(\theta, x), \hat v^m_2(\theta, x))  \geq F(x,  \hat{\Bar{v}}^m_1(\theta, x), \hat v^m_2(\theta, x)) , \ m \geq 1.
\]
Hence from (\ref{eq1usc}) and (\ref{eq2usc}), we obtain
\[
F(x,  \hat {\tilde{v}}_1(\theta, x), \hat v_2(\theta, x))  \geq F(x,  \hat{\Bar{v}}_1(\theta, x), \hat v_2(\theta, x)) , \ \hat {\tilde{v}}_1 \in \hat S_1.
\]
Hence $\hat{\Bar{v}}_1 \in H_1 (\hat v_2)$. By a symmetric argument, one can show that for 
$\hat{\Bar{v}}^m_2 \in H_2(\hat v^m_1)$ and any limit point $\hat{\Bar{v}}_2$ of $\{ \hat{\Bar{v}}^m_2 \}$, we have
$\hat{\Bar{v}}_2 \in H_2(\hat v_1)$. This proves that the map is u.s.c. 
\end{proof}
\vspace{0.1in} 
\begin{theorem}\label{Nashequilibriumdiscounted} Assume (A1) and (A2). Then there exists an $\alpha$-discounted
Nash equilibrium in  $\hat{\mathcal S}_1 \times \hat{\mathcal S}_2$.
\end{theorem}
\begin{proof} From Lemma \ref{usc1} and Fan's fixed point theorem \cite{Fan}, there exists a fixed point 
$(\hat v^*_1, \hat v^*_2) \in \hat{\mathcal S}_1 \times \hat{\mathcal S}_2$,
for the map $(\hat v_1, \hat v_2) \mapsto H(\hat v_1,\hat v_2) $ from $\hat{\mathcal S}_1 \times \hat{\mathcal S}_2 \to 
2^{\hat{\mathcal S}_1} \times 2^{\hat{\mathcal S}_2}$, i.e., 
\[
(\hat v^*_1, \hat v^*_2) \in H (\hat v^*_1, \hat v^*_2) .
\]
This implies that $(\psi^{\hat v^*_2}_{\alpha, 1}, \psi^{\hat v^*_1}_{\alpha, 2})$ satisfies 
\interdisplaylinepenalty=0
\begin{eqnarray}\label{coupleddiscountedhjb}
\alpha \theta \frac{\partial \psi^{\hat v^*_2}_{\alpha, 1}}{\partial \theta} & = & 
\inf_{v_1 \in V_1} \Big[ \langle b(x, v_1, \hat v^*_2(\theta, x)), \nabla_x \psi^{\hat v^*_2}_{\alpha, 1} \rangle 
+ \theta r_1(x, v_1, \hat v^*_2(\theta, x)) \psi^{\hat v^*_2}_{\alpha, 1} \Big] \nonumber \\ \nonumber 
&& + \frac{1}{2} \ {\rm trace}(a(x) \nabla^2_x \psi^{\hat v^*_2}_{\alpha, 1})  \\ \nonumber
& = & 
\langle b(x, \hat v^*_1(\theta, x), \hat v^*_2(\theta, x)), \nabla_x \psi^{\hat v^*_2}_{\alpha, 1} \rangle 
+ \theta r_1(x, \hat v^*_1(\theta, x), \hat v^*_2(\theta, x)) \psi^{\hat v^*_2}_{\alpha, 1} \Big] \\ \nonumber
&& + \frac{1}{2} \ {\rm trace}(a(x) \nabla^2_x \psi^{\hat v^*_2}_{\alpha, 1}) , \\ 
&& \psi^{\hat v^*_2}_{\alpha, 1} (0, x) \ = \  1, \theta \in (0, \ \Theta) , \ x \in \mathbb{R}^d \\  \nonumber 
\alpha \theta \frac{\partial \psi^{\hat v^*_1}_{\alpha, 2}}{\partial \theta} & = & 
\inf_{v_2 \in V_2} \Big[ \langle b(x, \hat v^*_1(\theta, x), v_2), \nabla_x \psi^{\hat v^*_1}_{\alpha, 2} \rangle 
+ \theta r_2(x, \hat v^*_1(\theta, x), v_2) \psi^{\hat v^*_1}_{\alpha, 2} \Big]   \\ \nonumber
&& + \frac{1}{2} \ {\rm trace}(a(x) \nabla^2_x \psi^{\hat v^*_1}_{\alpha, 2})  \nonumber \\ \nonumber
& = & 
\langle b(x, \hat v^*_1(\theta,  x), \hat v^*_2(\theta, x)), \nabla_x \psi^{\hat v^*_1}_{\alpha, 2} \rangle 
+ \theta r_2(x, \hat v^*_1\theta, (x), \hat v^*_2(\theta, x) \psi^{\hat v^*_1}_{\alpha, 2} \Big]  \\  \nonumber 
&& + \frac{1}{2} \ {\rm trace}(a(x) \nabla^2_x \psi^{\hat v^*_1}_{\alpha, 2}) , \\ \nonumber 
\psi^{\hat v^*_1}_{\alpha, 2} (0, x) & = & 1, \theta \in (0, \ \Theta) , \ x \in \mathbb{R}^d .
\end{eqnarray}
Now using the representation of the solutions to the p.d.e.s from Theorem \ref{pdeexistence}, it follows that
\interdisplaylinepenalty=0
\begin{eqnarray*}
\psi^{\hat v^*_2}_{\alpha, 1} (\theta, x) & = & \inf_{v_1 \in {\mathcal M}_1} J^{v_1, \hat v^*_2}_{\alpha, 1}(\theta, x)\\
& = & J^{\hat v^*_1, \hat v^*_2}_{\alpha, 1}(\theta, x), \\
\psi^{\hat v^*_1}_{\alpha, 2} (\theta, x) & = & \inf_{v_2 \in {\mathcal M}_2} J^{\hat v^*_1, v_2}_{\alpha, 2}(\theta, x)\\
& = & J^{\hat v^*_1, \hat v^*_2}_{\alpha, 2}(\theta, x).
\end{eqnarray*}
Thus, we obtain
\interdisplaylinepenalty=0
\begin{eqnarray*}
 J^{v_1,\hat v^*_2}_{\alpha, 1}(\theta, x) & \geq  & J^{\hat v^*_1, \hat v^*_2}_{\alpha, 1}(\theta, x), \ \forall \ 
v_1 \in {\mathcal M}_1 , \\
 J^{\hat v^*_1, v_2}_{\alpha, 2}(\theta, x) & \geq  & J^{\hat v^*_1, \hat v^*_2}_{\alpha, 2}(\theta, x), \ \forall \ 
v_2 \in {\mathcal M}_2.
\end{eqnarray*}
This proves the existence of a Nash equilibrium which is a pair of eventually stationary Markov strategies.
\end{proof}

\section{Analysis of Ergodic Cost Criterion} In this section, we study the ergodic cost criterion for the 
risk-sensitive stochastic dynamic games.  Since cost criterion is dependent on the limiting behavior of the 
state dynamics, it is quite natural to assume certain stable behavior of the dynamics.
We assume the following Lyapunov type stability condition:
 
\vspace{.1in}

\noindent {\bf (A3)} There exists constants $\delta > 0, c > 0$, a compact set $C$ of $\mathbb{R}^d$
and $W \in C^2(\mathbb{R}^d)$ such that
\begin{eqnarray*}
W & \geq & 1 \\
{\mathcal L} W(x, v_1, v_2) & \leq & - 2 \delta W(x) + c I_C (x), \ \forall \ x \in \mathbb{R}^d, \, 
v_1 \in V_1, v_2 \in V_2,
\end{eqnarray*}
where for $\ x \in \mathbb{R}^d, \, v_1 \in V_1, v_2 \in V_2, \varphi \in C^2(\mathbb{R}^d)$,
\begin{equation}\label{generator-state}
{\mathcal L} \varphi (x, v_1, v_2) \ = \ \langle b(x, v_1, v_2), \nabla_x \varphi(x) \rangle + \frac{1}{2} \,
{\rm trace}(a(x) \nabla^2_x \varphi(x)) .
\end{equation}

We also assume the following technical assumption.

\vspace{.1in}
\noindent {\bf (A4)} (Small cost condition) 
\[
\theta \|r_k\|_{\infty} \leq \delta, \ k = 1,2 , \ \theta \in (0, \Theta ).
\]
Now we establish some properties of the controlled diffusion (\ref{statedynamics}) needed later.
 Set
\begin{equation}\label{aset}
C_0 \ = \ \{ x \in \mathbb{R}^d | W(x) > 1 + \frac{c}{\delta} \}.
\end{equation}
Note that $C_0$ is the complement of a compact set in $\mathbb{R}^d$. 

\vspace{.1in}
\begin{lemma}\label{expobound} Assume (A1) and (A3). Let $X$ be the process (\ref{statedynamics}) corresponding to $(v_1, v_2) \in {\mathcal M}_1 \times {\mathcal M}_2$. Then for each $y \in C_0$ and $r > 0$ 
such that $B(y, r) \subseteq C_0$, we have
\[
E^{v_1,v_2}_x [ e^{\delta \tau^r_y} ] \ \leq \ W(x) , \ x \in \mathbb{R}^d,
\]
where $\tau^r_y = \inf \{ t \geq 0 | X(t) \in B(y, r) \}$.
\end{lemma}
\begin{proof} Fix $x \in \mathbb{R}^d$ and $R > 0$ large. Let
\[
\tau_R \ = \ \inf \{ t \geq 0 | \| X(t) \| \geq R \}.
\]
Using It$\hat{\rm o}$-Dynkin formula, we obtain
\interdisplaylinepenalty=0
\begin{eqnarray*}
E^{v_1, v_2}_x \Big[ e^{\delta ( \tau^r_y \wedge \tau_R )} W(X(\tau^r_y \wedge \tau_R )) \Big]
& = &  E^{v_1, v_2}_x \Big[ \int^{\tau^r_y \wedge \tau_R }_0 e^{\delta s} [{\mathcal L}
W(X(s), v_1(s, X(s)), v_2(s, X(s)))\\
&& + \delta W(X(s))] ds \Big] + W(x) \\
& \leq & (-1 + \frac{c}{\delta}) E^{v_1, v_2}_x \Big[ e^{\delta(\tau^r_y \wedge \tau_R) } \Big] + W(x)\\
& \leq & \frac{c}{\delta} E^{v_1, v_2}_x \Big[ e^{\delta \tau^r_y } \Big].
\end{eqnarray*}
Now by letting $R \to \infty$, using Fatou's lemma, we obtain
\[
E^{v_1, v_2}_x \Big[ e^{\delta  \tau^r_x} W(X(\tau^r_y)) \Big] \leq 
\frac{c}{\delta} E^{v_1, v_2}_x \Big[ e^{\delta \tau^r_y } \Big] + W(x).
\]
Since $B(y, r) \subseteq C_0$, we have $W \geq 1 + \frac{c}{\delta}$ on $B(y, r)$. The desired estimate then follows.
\end{proof}

\vspace{.1in}
The following lemma follows by an application of It$\hat{\rm o}$-Dynkin formula applied to
$e^{\theta \int^t_0 r_k(X(s), v_1(s, X(s)), v_2(s, X(s))) ds} W(X(t)), 0 \leq t \leq \tau_R \wedge T$
 and using (A4).

\vspace{.1in}
\begin{lemma}\label{ergodicbound} Assume (A1), (A3) and (A4). Let $X$ be the process 
(\ref{statedynamics}) corresponding to
$(v_1, v_2) \in {\mathcal M}_1 \times {\mathcal M}_2$. Then for $k=1,2$, we have
\[
\begin{array}{ll}
E^{v_1, v_2}_x \Big[ e^{\theta \int^T_0 r_k(X(t), v_1(t, X(t)), v_2(t, X(t))) dt} W(X(T)) \Big] & \\
\ \ \ \  \leq \ [W(x) + c T] \, 
E^{v_1, v_2}_x \Big[ e^{\theta \int^T_0 r_k(X(t), v_1(t, X(t)), v_2(t, X(t))) dt}  \Big] , \ x \in \mathbb{R}^d, 
T > 0. & 
\end{array}
\]
\end{lemma}


We now establish the following result which plays a crucial role in what follows.
\vspace{.1in}
\begin{theorem}\label{thm4.1} Assume (A1) and (A2). There exists $(\hat v^*_{1, \alpha}, \hat v^*_{2, \alpha}) \in 
\hat{\mathcal S}_1 \times  \hat{\mathcal S}_2$ such that 
$(\psi^{\hat v^*_{2,\alpha}}_{\alpha, 1}, \psi^{\hat v^*_{1, \alpha}}_{\alpha, 2})$ satisfies
\interdisplaylinepenalty=0
\begin{eqnarray*}
\alpha \theta \frac{\partial \psi^{ \hat v^*_{2,\alpha}}_{\alpha, 1}}{\partial \theta} & = & 
\inf_{v_1 \in V_1} \Big[ \langle b(x, v_1, \hat v^*_{2,\alpha}(\theta, x)), 
\nabla_x \psi^{\hat v^*_{2,\alpha}}_{\alpha, 1} \rangle 
+ \theta r_1(x, v_1,\hat v^*_{2,\alpha}(\theta, x)) \psi^{ \hat v^*_{2,\alpha}}_{\alpha, 1} \Big] \\ 
&& + \frac{1}{2} \ {\rm trace}(a(x) \nabla^2_x \psi^{\hat v^*_{2,\alpha}}_{\alpha, 1})  \\ 
& = & 
\langle b(x, \hat v^*_{1,\alpha}(\theta, x), \hat v^*_{2, \alpha}(\theta, x)), \nabla_x 
\psi^{ \hat v^*_{2,\alpha}}_{\alpha, 1}\rangle 
+ \theta r_1(x, \hat v^*_{1,\alpha}(\theta, x),\hat v^*_{2, \alpha}(\theta, x)) \psi^{v^*_{2,\alpha}}_{\alpha, 1} \Big] \\ 
&& + \frac{1}{2} \ {\rm trace}(a(x) \nabla^2_x \psi^{ \hat v^*_{2,\alpha}}_{\alpha, 1}) , \\
\psi^{\hat v^*_{2,\alpha}}_{\alpha, 1} (0, x) & = & 1, \theta \in (0, \ \Theta) , \ x \in \mathbb{R}^d , \\
\alpha \theta \frac{\partial \psi^{\hat v^*_{1, \alpha}}_{\alpha, 2}}{\partial \theta} & = & 
\inf_{v_2 \in V_2} \Big[ \langle b(x, \hat v^*_{1}(\theta, x), v_2), \nabla_x \psi^{\hat v^*_{1, \alpha}}_{\alpha, 2} \rangle 
+ \theta r_2(x, \hat v^*_{1, \alpha}(\theta, x), v_2) \psi^{\hat v^*_{1, \alpha}}_{\alpha, 2} \Big]  \\ 
&& + \frac{1}{2} \ {\rm trace}(a(x) \nabla^2_x \psi^{\hat v^*_{1, \alpha}}_{\alpha, 2})  \\ 
& = & 
\langle b(x, \hat v^*_{1,\alpha}(\theta,  x),\hat v^*_{2,\alpha}(\theta, x)), \nabla_x 
\psi^{\hat v^*_{1, \alpha}}_{\alpha, 2} \rangle 
+ \theta r_2(x, \hat v^*_{1,\alpha}\theta, (x), \hat v^*_{2,\alpha}(\theta, x) \psi^{v^*_{1, \alpha}}_{\alpha, 2} \Big]  \\ 
&& + \frac{1}{2} \ {\rm trace}(a(x) \nabla^2_x \psi^{\hat v^*_{1, \alpha}}_{\alpha, 2}) , \\ 
\psi^{\hat v^*_{1, \alpha}}_{\alpha, 2} (0, x) & = & 1, \theta \in (0, \ \Theta) , \ x \in \mathbb{R}^d .
\end{eqnarray*}
Moreover, the following estimates holds:
\interdisplaylinepenalty=0
\begin{eqnarray}\label{esimateuniform1}
\frac{1}{\psi^{ \hat v^*_{2,\alpha}}_{\alpha, 1}(\theta, x)} 
\frac{\partial \psi^{ \hat v^*_{2, \alpha}}_{\alpha, 1}(\theta, x)}{\partial \theta} & \leq & 
\frac{\theta \|r_1\|_{\infty} }{\alpha}, \\ \nonumber
\frac{1}{\psi^{ \hat v^*_{1,\alpha}}_{\alpha, 2}(\theta, x)} 
\frac{\partial \psi^{ \hat v^*_{1,\alpha}}_{\alpha, 2}(\theta, x)}{\partial \theta} & \leq & 
\frac{\theta \|r_2\|_{\infty} }{\alpha},
\end{eqnarray}
for all $ \alpha > 0, \theta \in (0, \ \Theta),  x \in \mathbb{R}^d.$
\end{theorem}

\begin{proof} The first part of the proof follows from Theorem \ref{Nashequilibriumdiscounted}.
Using Theorem \ref{pdeexistence} (ii), we have
\[
\psi^{\hat v^*_{2,\alpha}}_{\alpha, 1}(\theta, x) \ = \ \inf_{v_1 \in {\mathcal M}_1} E^{v_1, \hat v^*_{2, \alpha}}_x 
\Big[ e^{\theta \int^\infty_0 e^{-\alpha t} r_1(X(t), v_1(t, X(t)), \hat v^*_{2, \alpha}(\theta e^{-\alpha t}, X(t))) dt}
\Big].
\]
Now using the envelope theorem, see \cite{MS},
we have 
\interdisplaylinepenalty=0
\begin{eqnarray*}
\frac{\psi^{ \hat v^*_{2,\alpha}}_{\alpha, 1}(\theta, x)}{\partial \theta} & = & 
 E^{v^*_1, \hat v^*_{2, \alpha}}_x 
\Big[ \theta\int^\infty_0 e^{-\alpha t} r_1(X(t), v^*_1(t, X(t)), \hat v^*_{2, \alpha}(\theta e^{-\alpha t}, X(t))) dt \\
&& \times e^{\theta \int^\infty_0 e^{-\alpha t} r_1(X(t), v^*_1(t, X(t)),   \hat v^*_{2, \alpha}
(\theta e^{-\alpha t}, X(t))) dt} \Big] \\
& \leq & \frac{\theta \|r_1\|_{\infty} }{\alpha} \psi^{\hat v^*_{2,\alpha}}_{\alpha, 1}(\theta, x),
\end{eqnarray*}
where $v^*_1$ is a minimizer for $ J^{v_1, v^*_{2, \alpha}}(\theta, x)$ over $v_1 \in {\mathcal M}_1$.
This completes the proof of first estimate. The proof of the second estimate is symmetric.
\end{proof}

\vspace{.1in}
Fix $x_0 \in C_0$ and set
\interdisplaylinepenalty=0
\begin{equation}\label{barpsi}
\Bar{\psi}^{\hat v^*_{2,\alpha}}_{\alpha, 1}(\theta, x) \ = \ 
\frac{\psi^{\hat v^*_{2,\alpha}}_{\alpha, 1}(\theta, x)}{\psi^{ \hat v^*_{2,\alpha}}_{\alpha, 1}(\theta, x_0)}, \ 
\Bar{\psi}^{ \hat v^*_{1,\alpha}}_{\alpha, 2}(\theta, x) \ = \ 
\frac{\psi^{\hat v^*_{1,\alpha}}_{\alpha, 2}(\theta, x)}{\psi^{ \hat v^*_{1,\alpha}}_{\alpha, 2}(\theta, x_0)}, x \in 
\mathbb{R}^d.
\end{equation}
\begin{lemma}\label{lemmabound} Assume (A1), (A3) and (A4). Then
\begin{eqnarray*}
\Bar{\psi}^{\hat v^*_{2,\alpha}}_{\alpha, 1}(\theta, x) &\leq & W(x) \\
\Bar{\psi}^{\hat v^*_{1,\alpha}}_{\alpha, 2}(\theta, x) &\leq & W(x), \ x \in \mathbb{R}^d.
\end{eqnarray*}
\end{lemma}
\begin{proof} Using It$\hat{\rm o}$-Dynkin formula it is easy to see that for each $v_1 \in {\mathcal M}_1$, and $r > 0$
small enough, we have
\interdisplaylinepenalty=0
\begin{eqnarray*}
\psi^{ \hat v^*_{2,\alpha}}_{\alpha, 1}(\theta, x)& \leq & E^{v_1, \hat v^*_{2, \alpha}}_x \Big[
e^{\theta  \int^{\tau^r_{x_0}}_0 e^{-\alpha t} r_1(X(t), v_1(t, X(t)), \hat v^*_{2, \alpha}(\theta e^{-\alpha t},X(t))) dt} 
\psi^{ \hat v^*_{2,\alpha}}_{\alpha, 1}(\theta e^{-\alpha \tau^r_{x_0}}, X(\tau^r_{x_0})) \Big] \\
& \leq & \sup_{y \in \partial B(x_0, r)} \psi^{\hat v^*_{2,\alpha}}_{\alpha, 1}(\theta, y) \, 
E^{v_1, \hat v^*_{2, \alpha}}_x \Big[e^{\theta  \int^{\tau^r_{x_0}}_0 
e^{-\alpha t} r_1(X(t), v_1(t, X(t)), \hat v^*_{2, \alpha}(\theta e^{-\alpha t},X(t))) dt} \Big] \\
& \leq & \sup_{y \in \partial B(x_0, r)} \psi^{ \hat v^*_{2,\alpha}}_{\alpha, 1}(\theta, y)
E^{v_1, \hat v^*_{2, \alpha}}_x \Big[ e^{\delta \tau^r_{x_0}} \Big] .
\end{eqnarray*}
The second inequality follows from the fact that 
$ \theta \mapsto \psi^{ v^*_{2,\alpha}}_{\alpha, 1}(\theta, x)$ is non-decreasing in $\theta$ for each fixed $x$. Now using Lemma \ref{expobound}, we obtain
\begin{eqnarray*}
\psi^{ \hat v^*_{2,\alpha}}_{\alpha, 1}(\theta, x)& \leq &  \sup_{y \in \partial B(x_0, r)} \psi^{ \hat v^*_{2,\alpha}}_{\alpha, 1}(\theta, y) W(x).
\end{eqnarray*}
Now by letting $r \downarrow 0$, we obtain
\[
\psi^{\hat v^*_{2,\alpha}}_{\alpha, 1}(\theta, x) \leq \psi^{ \hat v^*_{2,\alpha}}_{\alpha, 1}(\theta, x_0) W(x).
\]
Hence the first estimate follows. The proof of the second estimate is similar.
\end{proof}
\vspace{.1in}
\begin{lemma} Assume (A1), (A3) and (A4). Then we have the following 
\begin{eqnarray*}
\Bar{\psi}^{\hat v^*_{2,\alpha}}_{\alpha, 1}(\theta, x) &\geq & \frac{1}{W(x_0)} , \\
\Bar{\psi}^{\hat v^*_{1,\alpha}}_{\alpha, 2}(\theta, x) &\geq & \frac{1}{W(x_0)}, \ x \in C_0.
\end{eqnarray*}
\end{lemma}
\begin{proof} For $x \in C_0$, using It$\hat{\rm o}$-Dynkin formula, it follows that for each $v_1 \in {\mathcal M}_1$, and $r > 0$
small enough
\interdisplaylinepenalty=0
\begin{eqnarray*}
\psi^{ \hat v^*_{2,\alpha}}_{\alpha, 1}(\theta, x_0)& \leq & E^{v_1, \hat v^*_{2, \alpha}}_{x_0} \Big[
e^{\theta  \int^{\tau^r_{x}}_0 e^{-\alpha t} r_1(X(t), v_1(t, X(t)), \hat v^*_{2, \alpha}(\theta e^{-\alpha t},X(t))) dt} 
\psi^{ \hat v^*_{2,\alpha}}_{\alpha, 1}(\theta e^{-\alpha \tau^r_{x}}, X(\tau^r_{x})) \Big] \\
& \leq & \sup_{y \in \partial B(x, r)} \psi^{\hat v^*_{2,\alpha}}_{\alpha, 1}(\theta, y) \, 
E^{v_1, \hat v^*_{2, \alpha}}_{x_0} \Big[e^{\theta  \int^{\tau^r_{x}}_0 
e^{-\alpha t} r_1(X(t), v_1(t, X(t)), \hat v^*_{2, \alpha}(\theta e^{-\alpha t},X(t))) dt} \Big].
\end{eqnarray*} 
Hence
\interdisplaylinepenalty=0
\begin{eqnarray*}
\frac{\psi^{ \hat v^*_{2,\alpha}}_{\alpha, 1}(\theta, x_0)}{\displaystyle{\sup_{y \in \partial B(x, r)}
 \psi^{\hat v^*_{2,\alpha}}_{\alpha, 1}(\theta, y)}} & \leq & E^{v_1, \hat v^*_{2, \alpha}}_{x_0} 
\Big[e^{\theta  \int^{\tau^r_{x}}_0 e^{-\alpha t} r_1(X(t), v_1(t, X(t)), \hat v^*_{2, \alpha}(\theta e^{-\alpha t},X(t))) dt} \Big]\\
& \leq & E^{v_1, \hat v^*_{2, \alpha}}_{x_0} [e^{\theta \|r_1\|_{\infty} \tau^r_x}] \\
& \leq & W(x_0).
\end{eqnarray*}
The last inequality follows  using Lemma \ref{expobound} and the fact that $x \in C_0$. Now by letting $r \downarrow 0$,
the first estimate follows. The proof of the second lower bound follows by a symmetric argument.
\end{proof}
\vspace{.1in}
\begin{theorem} Assume (A1)-(A4). For each $\theta \in (0, \Theta)$, there exists $(v^*_{1}, v^*_{2}) \in 
{\mathcal S}_1 \times  {\mathcal S}_2$ and 
$(\rho_1, \psi_1) , (\rho_2, \psi_2) 
\in \mathbb{R} \times W^{2, p}_{loc}(\mathbb{R}^d) \cap O(W)$ satisfying
\interdisplaylinepenalty=0
\begin{eqnarray}\label{coupledergodichjb}
 \theta \rho_1  \psi_1 & = & 
\inf_{v_1 \in V_1} \Big[ \langle b(x, v_1, v^*_{2}( x)), 
\nabla_x \psi_1 \rangle  + \theta r_1(x, v_1, v^*_{2}( x)) \psi_ 1 \Big] \nonumber \\ \nonumber
&& + \frac{1}{2} \ {\rm trace}(a(x) \nabla^2_x \psi_1)  \\ \nonumber 
& = & 
\langle b(x, v^*_1( x), v^*_{2}( x)), \nabla_x  \psi_1\rangle 
+ \theta r_1(x, v^*_{1}( x), v^*_{2}(x)) \psi_1 \Big] \\ \nonumber 
&& + \frac{1}{2} \ {\rm trace}(a(x) \nabla^2_x \psi_ 1) , \\ \nonumber 
\psi_1( x_0) & = & 1,  \ x \in \mathbb{R}^d, \\
\theta \rho_2  \psi_2 & = & 
\inf_{v_2 \in V_2} \Big[ \langle b(x, v^*_1(x), v_{2}), \nabla_x \psi_2 \rangle 
+ \theta r_2(x, v^*_1(x), v_{2}) \psi_ 2 \Big] \\ \nonumber 
&& + \frac{1}{2} \ {\rm trace}(a(x) \nabla^2_x \psi_2)  \\ \nonumber 
& = & 
\langle b(x, v^*_1( x), v^*_{2}( x)), \nabla_x  \psi_2\rangle 
+ \theta r_2(x, v^*_{1}( x), v^*_{2}(x)) \psi_1 \Big] \\ \nonumber
&& + \frac{1}{2} \ {\rm trace}(a(x) \nabla^2_x \psi_ 2) , \\ \nonumber 
\psi_2( x_0) & = & 1, \ x \in \mathbb{R}^d.
\end{eqnarray}
\end{theorem}
\begin{proof} From Theorem \ref{thm4.1}, it follows that $\Bar{\psi}^{\hat v^*_{2, \alpha}}_{\alpha, 1} \in 
W^{1, 2, p}_{{\rm loc}}((0, \Theta) \times \mathbb{R}^d), p \geq d+1, \alpha > 0$, satisfies the 
p.d.e.
\interdisplaylinepenalty=0
\begin{equation}\label{pdeauxillary1}
\eta_{\alpha}(\theta, x) \Bar{\psi}^{\hat v^*_{2, \alpha}}_{\alpha, 1} = \ 
\inf_{v_1 \in V_1} \Big[ {\mathcal L} \Bar{\psi}^{\hat v^*_{2, \alpha}}_{\alpha, 1}(x, v_1, \hat v^*_{2, \alpha}(\theta,x))
+ \theta r_1(x, v_1, \hat v^*_{2, \alpha}(x)) \Bar{\psi}^{\hat v^*_{2, \alpha}}_{\alpha, 1}\Big],
\end{equation}
where
\begin{equation}\label{eq1theorem4.2}
\eta_{\alpha}(\theta, x) = \alpha \theta \frac{1}{\psi^{\hat v^*_{2, \alpha}}_{\alpha, 1}}
\frac{\partial \psi^{\hat v^*_{2, \alpha}}_{\alpha, 1}}{\partial \theta}(\theta, x).
\end{equation}
From the estimates in Theorem \ref{thm4.1} and Lemma \ref{lemmabound}, it follows that the l.h.s. of (\ref{pdeauxillary1})
is locally uniformly bounded in $\alpha > 0$. Hence by freezing the l.h.s., using the arguments in
[\cite{Borkar}, p.158], it follows that 
\begin{equation}\label{estimate1theorem4.2}
\|\Bar{\psi}^{\hat v^*_{2, \alpha}}_{\alpha , 1} \|_{2, p; B_R} \ \leq \ K,
\end{equation}
where $K >0$ is a constant independent of $\alpha > 0$. 
Also note that 
\interdisplaylinepenalty=0
\begin{equation}\label{eq2theorem4.2}
\frac{\partial \Bar{\psi}^{\hat v^*_{2, \alpha}}_{\alpha, 1}}{\partial \theta}(\theta, x) \ = \ 
\Bar{\psi}^{\hat v^*_{2, \alpha}}_{\alpha, 1} (\theta, x)
\Big[ \frac{1}{\psi^{\hat v^*_{2, \alpha}}_{\alpha,1}(\theta, x)} 
\frac{\partial \psi^{\hat v^*_{2, \alpha}}_{\alpha, 1}}{\partial \theta}
 (\theta, x) - \frac{1}{\psi^{\hat v^*_{2, \alpha}}_{\alpha,1}(\theta, x_0)}
 \frac{\partial \psi^{\hat v^*_{2, \alpha}}_{\alpha, 1}}{\partial \theta} (\theta, x_0)\Big].
\end{equation}
Now combining the estimates in Theorem \ref{thm4.1}, Lemma \ref{lemmabound}, (\ref{estimate1theorem4.2}) and the
identity (\ref{eq2theorem4.2}), it follows that
\begin{equation}\label{estimate2theorem4.2}
\|\Bar{\psi}^{\hat v^*_{2, \alpha}}_{\alpha , 1} \|_{1,2, p; (0, \Theta) \times B_R} \ \leq \ K,
\end{equation}
where $K >0$ is a constant independent of $\alpha > 0$. 
Hence $\{\Bar{\psi}^{\hat v^*_{2, \alpha}}_{\alpha , 1} | \alpha > 0\}$ is  weakly compact in 
$W^{1,2,p}((0, \Theta) \times B_R), R > 0, p \geq d+1$.

Now by a diagonalization argument, there exists $\psi_1 \in W^{1,2,p}_{{\rm loc}}((0, \Theta) \times 
\mathbb{R}^d), p \geq d +1$ and a subsequence $\alpha_n \downarrow 0$ such that 
\begin{equation}\label{eq3theorem4.2}
\Bar{\psi}^{\hat v^*_{2, \alpha_n}}_{\alpha_n, 1} \to \psi_1 \ {\rm weakly\ in}\ 
W^{1,2,p}((0, \Theta) \times B_R), \ \forall \ R > 0.
\end{equation}
By closely mimicking the arguments  in \cite{AnupBorkarSuresh}, it follows that 
\begin{equation}\label{eq4theorem4.2}
\eta_{\alpha_n}\to \rho_1(\theta)  \ {\rm weakly\ in}\ 
L^{\infty} ((0, \Theta) \times \mathbb{R}^d). 
\end{equation}
Since $\hat{\mathcal S}_1 \times \hat{\mathcal S}_2$ is compact, it follows that there exists 
$(v^*_1, v^*_2) \in \hat{\mathcal S}_1 \times \hat{\mathcal S}_2$ such that 
along a subsequence (by an abuse of notation $\alpha_n$ itself) 
$(\hat v^*_{\alpha_n, 1}, \hat v^*_{\alpha_n, 2}) \to (v^*_1, v^*_2)$ in $\hat{\mathcal S}_1 \times \hat{\mathcal S}_2$. 
Now by letting $\alpha_n \downarrow 0$ in (\ref{pdeauxillary1}), it follows that $(\rho_1, \psi_1)
\in \mathbb{R} \times W^{1,2,p}_{{\rm loc}}((0, \Theta) \times \mathbb{R}^d)$ is a solution to 
the first part of the system of p.d.e.s in (\ref{coupledergodichjb}). Clearly $(v^*_1, v^*_2) \in {\mathcal S}_1 \times {\mathcal S}_2$. From  Lemma \ref{lemmabound}, it follows that
$\psi_1 \in O(W)$. The proof for the second part of the p.d.e. is similar. 
\end{proof}

\vspace{0.1in}

 To proceed further we  assume that:

\noindent {\bf (A5)} There exists $\beta > 1$ and an inf compact $h : \mathbb{R}^d \to [0, \ \infty)$
such that 
\[
{\mathcal L} W^\beta (x, v_1, v_2) \leq - h(x) + \hat{c} I_{\hat{C}}, \forall \ x\in \mathbb{R}^d,
v_1 \in V_1, v_2 \in V_2,
\]
where $\hat c >0$ and $\hat C$ is a compact subset of $ \mathbb{R}^d$.

It is not difficult to see that if $W$ is a polynomial in $x$, then (A3) implies (A5), in particular, if
$W(x) = x Q x^{\perp}$ for some positive definite matrix $Q$, then (A5) follows from (A3).

Finally we prove the following result.

\begin{theorem} Assume (A1)-(A5). The stationary Markov strategies $(v^*_1, v^*_2)$ given in
Theorem 4.6 is a Nash equilibrium and $(\rho_1, \rho_2)$ is the corresponding 
Nash values.
\end{theorem}
\begin{proof} Let $X$ be the process (\ref{statedynamics}) corresponding to $(v_1, v^*_2), v_1 \in 
{\mathcal M}_1$ with initial condition $x$. Now using It$\hat{{\rm o}}$-Dynkin formula we obtain
\interdisplaylinepenalty=0
\[
\begin{array}{ll}
\displaystyle{
E^{v_1, v^*_2}_x \Big[ e^{\theta\int^{T \wedge \tau_R}_0 (r_1(X(t), v_1(t, X(t)), v^*_2(X(t))) - \rho_1) dt }
\psi_1(X(T \wedge \tau_R)) \Big]} &  \\
\displaystyle{ \  \  = \    E^{v_1, v^*_2}_x \Big[ \int^{T \wedge \tau_R}_0 
e^{\theta \int^t_0 (r_1(X(s), v_1(s, X(s)), v^*_2(X(s))) - \rho_1) ds } 
[{\mathcal L} \psi_1(X(t), v_1(t, X(t)), v^*_2(X(t))) }& \\
\ \ \ \  \  \  \  \  \  \  \ \  \  \  \ + \ \theta (r_1(X(t), v_1(t, X(t)), v^*_2(X(t))) - \rho_1)
\psi_1(X(t)) dt \Big] + \psi_1(x) & \\
\ \ \geq  \psi_1(x) . & 
\end{array}
\]
Hence
\interdisplaylinepenalty=0
\begin{eqnarray}\label{eq1theorem4.3}
\psi_1(x) & \leq & E^{v_1, v^*_2}_x \Big[ e^{\theta\int^{T \wedge \tau_R}_0 (r_1(X(t), v_1(t, X(t)), v^*_2(X(t))) -
 \rho_1) dt } W(X(T \wedge \tau_R)) \Big] \nonumber \\
& \leq & E^{v_1, v^*_2}_x \Big[ e^{\theta \int^{T \wedge \tau_R}_0 (r_1(X(t), v_1(t, X(t)), v^*_2(X(t))) -
 \rho_1) dt } W(X(T \wedge \tau_R)) I \{ T \leq \tau_R \} \Big] \\ \nonumber 
& & + e^{\theta (\|r_1\|_{\infty} - \rho_1)T} E^{v_1, v^*_2}_x \Big[W(X(T \wedge \tau_R)) I \{ T > \tau_R \}\Big].
\end{eqnarray}
Now using (A5), we get
\interdisplaylinepenalty=0
\begin{eqnarray*}
E^{v_1,v^*_2}_x \Big[ W(X(T \wedge \tau_R)) I \{T > \tau_R\} \Big] & \leq & 
\Big[ E^{v_1,v^*_2}_x\Big[ W^\beta(X(T \wedge \tau_R)) \Big] \Big]^{\frac{1}{\beta}} 
\Big[P^{v_1, v^*_2}_x (T > \tau_R) \Big]^{1 - \frac{1}{\beta}} \\
& \leq & (W^\beta (x) + \hat{c} T)^{\frac{1}{\beta}} 
\Big[P^{v_1, v^*_2}_x (T > \tau_R) \Big]^{1 - \frac{1}{\beta}},
\end{eqnarray*}
where $\beta > 1$ as in (A5) and the second inequality follows by (A5) by an application of 
It$\hat{\rm o}$-Dynkin formula to $W^\beta(X(t))$.  Hence
\begin{equation}\label{eq2theorem4.3}
\lim_{R \to \infty} E \Big[ W(X(T \wedge \tau_R)) I \{ T > \tau_R \} \Big] \ = \ 0.
\end{equation}
From (\ref{eq1theorem4.3}), (\ref{eq2theorem4.3}) and Lemma \ref{ergodicbound}, we get
\interdisplaylinepenalty=0
\begin{eqnarray}\label{eq3theorem4.3}
\psi_1(x) & \leq & E^{v_1, v^*_2}_x \Big[ e^{\theta\int^{T}_0 (r_1(X(t), v_1(t, X(t)), v^*_2(X(t)))- \rho_1) dt } 
W(X(T)) \Big] \nonumber \\
& \leq & e^{-\theta \rho_1 T} (W(x) + c T) 
E^{v_1, v^*_2}_x \Big[ e^{\theta \int^{T}_0 r_1(X(t), v_1(t, X(t)), v^*_2(X(t)))  dt }  \Big] .
\end{eqnarray}
Now by taking logarithm in (\ref{eq3theorem4.3}), then  divide by $\theta T$ and by letting 
$T \to \infty$ we get
\[
\rho_1 \leq \limsup_{T \to \infty} \frac{1}{\theta T} \log 
E^{v_1, v^*_2}_x \Big[ e^{\theta \int^{T}_0 r_1(X(t), v_1(t, X(t)), v^*_2(X(t)))  dt }  \Big] .
\]
i.e.,
\begin{equation}\label{eq4theorem4.3}
\rho_1 \leq \inf_{v_1 \in {\mathcal M}_1} \rho^{v_1, v^*_2}_1(\theta, x).
\end{equation}
Now let $X$ denote  the process (\ref{statedynamics}) corresponding to $(v^*_1, v^*_2)$ 
with initial condition $x$. Now using It$\hat{\rm o}$-Dynkin formula as above we get
\interdisplaylinepenalty=0
\begin{eqnarray*}
\psi_1(x) & = & E^{v^*_1, v^*_2}_x \Big[ e^{\theta\int^{T \wedge \tau_R}_0 (r_1(X(t), v^*_1( X(t)), v^*_2(X(t)))
 - \rho_1) dt } \psi_1(X(T \wedge \tau_R)) \Big] \\
& \geq & k E^{v^*_1, v^*_2}_x \Big[ e^{\theta\int^{T \wedge \tau_R}_0 (r_1(X(t), v^*_1(X(t)), v^*_2(X(t))) -
 \rho_1) dt }\Big] , 
\end{eqnarray*}
where
\[
k \ = \ \min \Big\{ \min_{y \in C^c_0} \psi_1(y), \, \frac{1}{W(x_0)} \Big\} > 0.
\]
Using Fatou's lemma, we obtain
\[
\psi_1(x) \geq k E^{v^*_1, v^*_2}_x \Big[ e^{\theta\int^{T}_0 (r_1(X(t), v^*_1(X(t)), v^*_2(X(t))) -
 \rho_1) dt }\Big]. 
\]
Hence it follows that
\[
\rho_1 \geq \limsup_{T \to \infty} \frac{1}{\theta T} \log 
E^{v^*_1, v^*_2}_x \Big[ e^{\theta \int^{T}_0 r_1(X(t), v^*_1( X(t)), v^*_2(X(t)))  dt }  \Big] .
\]
i.e.
\begin{equation}\label{eq5theorem3.4}
\rho_1 \geq \rho^{v^*_1, v^*_2}(\theta, x), \ x \in \mathbb{R}^d.
\end{equation}
Combining (\ref{eq4theorem4.3}) and (\ref{eq5theorem3.4}) we get
\[
\rho_1 = \rho_1^{v^*_1, v^*_2}(\theta, x) \leq \rho_1^{v_1, v^*_2}(\theta, x) \ \forall \ v_1 \in 
{\mathcal M}_1, \, x \in \mathbb{R}^d.
\]
A symmetric argument implies 
\[
\rho_2 = \rho_2^{v^*_1, v^*_2}(\theta, x) \leq \rho_2^{v^*_1, v_2}(\theta, x) \ \forall \ v_2 \in 
{\mathcal M}_2, \, x \in \mathbb{R}^d.
\]
This completes the proof.
\end{proof}

\section{Conclusion} In this paper we have established the existence of  a pair of Nash equilibrium strategies for risk sensitive stochastic games for diffusion process for two types of cost criteria. First we have established the the existence of  a pair of Nash equilibrium strategies for $\alpha$-discounted cost criterion in the class of eventually stationary Markov strategies under (A1) and (A2). For ergodic cost criterion we have proved the the existence of  a pair of Nash equilibrium strategies under a Lyapunov type stability assumption and a small cost condition and (A1), (A2). Both cases (A2) plays a crucial role in our analysis. It will be interesting to study the same problem without (A2). Finally we have chosen the same risk-sensitive parameter $\theta$ for both players for the sake of notational simplicity. The extension to the case when different players choose different risk-sensitive parameters would be routine. In fact our analysis extends to the case where player 1 is risk-averse and selects a risk-aversion parameter $\theta_1>0$ where as player 2 is risk-seeking and chooses a risk seeking parameter $\theta_2<0$. Finally we conclude with a remark on the zero-sum case. For the zero-sum case $$\bar{r}_1(x,u_1,u_2)+\bar{r}_2(x,u_1,u_2)=0$$ for all $x\in \mathbb{R}^d, \;u_1 \in U_1, \;u_2 \in U_2$. This does not imply that the sum of the risk-sensitive discounted (or ergodic) costs is zero. This is due to the multiplicative nature of the evaluation criterion. Thus in the zero-sum case, if we set $$r_1=-r_2=r,$$ then for $\theta>0$, player 1 is risk-averse whereas player 2 is risk-seeking. Thus the zero-sum case has to be necessarily  studied via Nash equilibria. The standard concepts like values, saddle-point equilibria etc will not be meaningful in the multiplicative cost criterion. In \cite{BG}, however a particular type of (non-standard) zero-sum risk-sensitive stochastic differential game has been studied which is not a special case of non-zero sum case studied in this paper. In \cite{BG} it is assumed that $r_1=r_2=r$. Player 1 tries to minimizes his risk-sensitive discounted (or ergodic) cost whereas player 2 tries to maximize the same. In this scenario, the notation of value, saddle-point equilibria etc make sense which have been analyzed through appropriate Hamilton-Jacobi-Isaacs (HJI) equations in \cite{BG}.

\end{document}